\documentclass{amsart}
\usepackage{amsmath,amssymb,euscript,graphicx,upgreek}
\usepackage[arrow,matrix,curve]{xy}
\newcommand\cA{\EuScript{A}}
\newcommand\cE{\EuScript{E}}

\newcommand\cG{\EuScript{G}}
\newcommand\cM{\EuScript{M}}
\newcommand\cR{\EuScript{R}}
\newcommand\cS{\EuScript{S}}
\newcommand\cU{\EuScript{U}}
\newcommand\cV{\EuScript{V}}

\newcommand{\bC}{\mathbb{C}}
\newcommand{\bR}{\mathbb{R}}
\newcommand{\bZ}{\mathbb{Z}}

\newcommand\ad{\operatorname{ad}}
\newcommand\Ad{\mathrm{Ad}}
\newcommand\Hol{\mathrm{Hol}}
\newcommand\Hom{\operatorname{Hom}}
\newcommand\End{\operatorname{End}}
\newcommand\ev{\operatorname{ev}}
\newcommand\rank{\operatorname{rank}}
\newcommand\res{\operatorname{res}}
\newcommand\tr{\operatorname{tr}}
\newcommand\hol{\operatorname{hol}}
\newcommand\coker{\operatorname{coker}}
\newcommand\PD{\operatorname{PD}}
\DeclareMathOperator{\im}{im}
\DeclareMathOperator{\id}{id}
\newcommand\pdr\partial
\newcommand\red{\mathrm{red}}

\newcommand\fla{\mathrm{flat}}

\newcommand\ZG{Z(G)}

\newcommand\fg{\mathfrak{g}}

\newcommand\oM{\overline M}
\newcommand\tM{\tilde{M}}
\newcommand\tP{\tilde{P}}
\newcommand\tD{\tilde{D}}
\newcommand\br{\tilde {\boldsymbol{r}}}
\newcommand\tgam{{\tilde\gam}}
\newcommand\tphi{{\tilde\phi}}
\newcommand\tPI{{\tilde\PI}}
\newcommand\Adp{{\Ad\circ\phi}}
\newcommand\Adtp{{\Ad\circ\tphi}}

\newcommand\al\alpha
\newcommand\gam\gamma
\newcommand\del\delta
\newcommand\eps\epsilon
\newcommand\sig\sigma
\newcommand\om\omega
\newcommand\upom\upomega
\newcommand\Gam\varGamma
\newcommand\Del\varDelta
\newcommand\PI\varPi
\newcommand\Om\varOmega

\newcommand\T{{\scalebox{.85}[.85]{$\scriptstyle\mathsf{T}$}}}
\newcommand{\sslash}{/\!\!/}
\newcommand\ii{\sqrt{-1}}

\newtheorem{thm}{Theorem}[section]
\newtheorem{lem}[thm]{Lemma}
\newtheorem{cor}[thm]{Corollary}
\newtheorem{pro}[thm]{Proposition}
\newtheorem{eg}[thm]{Example}
\numberwithin{equation}{section}

\newcommand\bra\langle
\newcommand\ket\rangle
\newcommand\two[4]{\Big\{\begin{array}{ll} {#1} & \mbox{if }{#2},  \\
   {#3} & \mbox{if } {#4}\end{array}}

\begin{document}

\begin{center}
{\bf\LARGE Conditions of smoothness of moduli spaces of flat connections\\
and of character varieties}
\vspace{1.5em}

Nan-Kuo Ho$^{1,a}$, Graeme Wilkin$^{2,b}$ and Siye Wu$^{1,c}$
\vspace{1ex}

{\small $^1\,$Department of Mathematics, National Tsing Hua University,
Hsinchu 30013, Taiwan\\

$^2\,$Department of Mathematics, National University of Singapore, Singapore
119076\\

$^a\,$E-mail address: {\tt nankuo@math.nthu.edu.tw}\\

$^b\,$E-mail address: {\tt graeme@nus.edu.sg}\\

$^c\,$E-mail address: {\tt swu@math.nthu.edu.tw}}
\end{center}
\vspace{1em}

\begin{quote}
{\bf Abstract.}
We use gauge theoretic and algebraic methods to examine sufficient conditions
for smooth points on the moduli space of flat connections on a compact
manifold and on the character variety of a finitely generated and presented
group.
We give a complete proof of the slice theorem for the action of the group of
gauge transformations on the space of flat connections.
Consequently, the slice is smooth if the second cohomology of the manifold
with coefficients in the semisimple part of the adjoint bundle vanishes.
On the other hand, we find that the smoothness of the slice for the character
variety of a finitely generated and presented group depends not only on the
second group cohomology but also on the relation module of the presentation.
However, when there is a single relator or if there is no relation among the
relators in the presentation, our condition reduces to the minimality of the
second group cohomology.
This is also verified using Fox calculus.
Finally, we compare the conditions of smoothness in the two approaches.

\end{quote}
\vspace{1em}

\section{Introduction}
Let $G$ be a reductive complex Lie group.
The moduli space of flat $G$-connections on a manifold $M$ is the quotient
of the space of flat reductive connections on principal $G$-bundles over $M$
modulo the group of gauge transformations.
It is well known that this moduli space can be identified with the character
variety of the fundamental group $\pi_1(M)$, i.e., the set of reductive
homomorphisms from $\pi_1(M)$ to $G$ modulo conjugations by $G$.

The motivation for using gauge theory is that one can compute information
about the deformation space in terms of the topology and geometry of the
bundle and the underlying manifold.
For example, on a Riemann surface the dimension of the moduli space follows
easily from the Riemann-Roch theorem (see for example \cite{AB}).
If the manifold is compact and K\"ahler, then Goldman and Millson \cite{GM}
were able to show that the singularities in the moduli space were quadratic.
In subsequent work, Simpson \cite{Sim} used the theory of Higgs bundles to
place restrictions on the fundamental group of a compact K\"ahler manifold.

On the other hand, one can study character varieties very concretely.
If $M$ is compact, then $\pi_1(M)$ is finitely generated and finitely
presented, and the character variety is the quotient of an affine variety
by $G$ acting on it algebraically.
Goldman \cite{G}, and later Goldman and Millson \cite{GM}, expressed
deformations in terms of group cohomology of $\pi_1(M)$.
This can be traced back to the work of Weil \cite{We} on deformations of
discrete subgroups of Lie groups and of Nijenhuis and Richardson
\cite{NR1,NR2} on deformations of graded Lie algebras and of homomorphisms
of Lie groups and Lie algebras.
A common ingredient in these studies is the cohomology of groups and algebras.

In this paper, we reexamine deformations in moduli spaces and in character 
varieties using the above two approaches, with an emphasis on finding
sufficient conditions for smooth points on these spaces.

In \S\ref{sec:gauge}, we study the local model and smoothness of the moduli
space of flat connections from the gauge theory construction.
Let $P$ be a principal $G$-bundle over a compact manifold $M$.
We give a complete proof of the slice theorem (Theorem~\ref{thm:slice}) for
the action on the space of flat connections on $P$ by the group $\cG(P)$ of
gauge transformations.
If a flat connection $D$ is good (i.e., reductive and having stabiliser $\ZG$),
then a neighbourhood of the gauge equivalence class $[D]$ in the moduli space
is homeomorphic to a neighbourhood of $D$ in the slice.
If in addition $H^2(M,\ad P)=H^2(M,Z(\fg))$, where $\fg$ is the Lie algebra of
$G$, then the moduli space is smooth at $[D]$ (Corollary~\ref{cor:mod}).
Attaining these results involves manipulation of infinite dimensional spaces
which has occurred in a number of contexts, for example, in the moduli of
(anti-)self-dual connections on four manifolds \cite{AHS,DK}, where the
structure group $G$ was assumed to be a compact Lie group.
However, in our situation, $G$ is complex reductive, and the action of $\cG(P)$
on the space of connections or the space of sections is not an isometry.
Instead, we adapt the proofs in \cite{Ki,Ko} for the moduli of holomorphic
bundles and of Hermitian-Einstein connections on K\"ahler manifolds to the
study of moduli of flat $G$-connections on Riemannian manifolds.
We include the proofs as we can not find them in the existing literature.
Finally, we compare the moduli spaces of flat connections on a non-orientable
manifold and of those on its orientation double cover
(Corollary~\ref{cor:nbhd}).

In \S\ref{sec:rep}, we study the same problems for the character variety
of any finitely generated and finitely presented group $\PI$, such as the
fundamental group of a compact manifold.
The space $\Hom(\PI,G)$ of homomorphisms from $\PI$ to $G$ is an affine
variety in the product of finite copies of $G$ defined by the relations
among the generators of $\PI$, and $G$ acts on it algebraically.
Therefore we can apply Luna's slice theorem \cite{Lu} to obtain similar
results as in the gauge theoretic approach.
Note that $H^2(\PI,\fg_\Adp)$ always contains $H^2(\PI,Z(\fg))$ and
$H^2(\PI,\fg_\Adp)=H^2(\PI,Z(\fg))$ is the condition under which any
infinitesimal deformation from $\phi$ can be integrated \cite{G}.
However, using the implicit function theorem, we found that
$\phi\in\Hom(\PI,G)$ is a smooth point on $\Hom(\PI,G)$ if the quantity
$\dim\Hom(\bar N,\fg_\Adp)^\PI-\dim H^2(\PI,\fg_\Adp)$ reaches its maximum at
$\phi$ (Corollary~\ref{cor:sm-hom}) instead of requiring the minimality of
$\dim H^2(\PI,\fg_\Adp)$ alone.
Here $\bar N$ is the relation module of the presentation of $\PI$ and
$\fg_\Adp$ is the Lie algebra $\fg$ on which $\PI$ acts via $\Adp$.
We then arrive at the same result in a different way: using Fox calculus and
its appearance in the free resolution of $\bZ$ by $\bZ\PI$-modules.
If there is a single relation among the generators of $\PI$ or if there are
no relations among the relators, then the above condition reduces to the
minimality of $\dim H^2(\PI,\fg_\Adp)$ (Proposition~\ref{pro:one}).
This is the case, for example, when $\PI$ is the fundamental group of
a compact (orientable or non-orientable) surface.
In this case, the moduli space of flat $G$-connections is smooth at $[\phi]$
if first, $\phi$ is reductive and its stabiliser is $\ZG$ (which implies
$H^0(\PI,\fg_\Adp)=Z(\fg)$) and second, $H^2(\PI,\fg_\Adp)=H^2(\PI,Z(\fg))$.
If the surface $M$ is orientable, then
$H^0(\PI,\fg_\Adp)\cong H^2(\PI,\fg_\Adp)$ by Poincar\'e duality, and hence
the conditions $H^0(\PI,\fg_\Adp)=Z(\fg)$ and $H^2(\PI,\fg_\Adp)=Z(\fg)$ are
equivalent \cite{G}.
But if $M$ is non-orientable, the two cohomology groups $H^0(\PI,\fg_\Adp)$
and $H^2(\PI,\fg_\Adp)$ are different and we need two conditions for
smoothness.
Furthermore, we find explicit formulas for both $H^0(\PI,\fg_\Adp)$ and
$H^2(\PI,\fg_\Adp)$ using Fox calculus (Proposition~\ref{pro:H02}).
We give an example to show that, in contrast to the case of a compact
orientable surface, points in $\Hom(\PI,G)$ with a stabiliser of minimal
dimension may not project to smooth points on the character variety.

In \S\ref{sec:compare}, we compare the gauge theoretic approach to moduli
spaces in \S\ref{sec:gauge} and the algebraic approach to character 
varieties in \S\ref{sec:rep}.
Using a spectral sequence, we find a relation between $H^2(M,\ad P)$ and
$H^2(\PI,\fg_\Adp)$, where $P$ is a flat $G$-bundle over $M$, $\PI=\pi_1(M)$,
and $\phi\in\Hom(\PI,G)$ is determined by the holonomy of $P$.
We find that the condition $H^2(M,\ad P)=H^2(M,Z(\fg))$ is stronger than
$H^2(\PI,\fg_\Adp)=H^2(\PI,Z(\fg))$ in general.
However, if $M$ is a compact (orientable or non-orientable) surface, then
$H^2(M,\ad P)=H^2(\PI,\fg_\Adp)$, and the sufficient conditions of smoothness
from the two approaches are identical.

In the Appendix, we present some examples of flat $G$-connections on
orientable and non-orientable manifolds (or homomorphisms from the fundamental
groups into $G$) that we referred to in the main text.

All classical complex Lie groups are reductive, and the complexification of a
compact Lie group is also complex reductive (cf.~\cite[Theorem~3, p.234]{Pr}).
Though we assumed that $G$ is a complex reductive Lie group throughout, our
methods clearly apply (with the exception of the use of Luna's slice theorem
in \S\ref{sec:luna}, which is not necessary when $G$ is compact) and our
results also hold when $G$ is a compact Lie group .

\medskip\noindent {\em Acknowledgments.}
The research of N.H.\ was supported by grant number 102-2115-M-007-003-MY2
from the National Science Council of Taiwan.
The research of G.W.\ was supported by grant number R-146-000-200-112 from
the National University of Singapore.
The research of S.W. was supported by grant number 105-2115-M-007-001-MY2
from the Ministry of Science and Technology of Taiwan.
The authors thank W.M.~Goldman for helpful discussions and the referee for
constructive comments.

\section{Gauge theory approach to moduli spaces}\label{sec:gauge}
\subsection{Preliminaries}
Let $M$ be a compact manifold of dimension $n$, $G$ be a complex reductive
group with Lie algebra $\fg$, and $P\to M$ be a principal $G$-bundle.
Let $\cA(P)$ be the set of $G$-connections on $P$ and $\cA^\fla(P)$ be the
subset of flat connections.
The group $\cG(P)$ of gauge transformations acts on $\cA(P)$ preserving
$\cA^\fla(P)$.
The centre $Z(G)$ of $G$ can be identified with a subgroup of $\cG(P)$
containing constant gauge transformations and it acts trivially on $\cA(P)$.

A connection $D$ on $P$ induces a connection $D^{\ad P}$ on the adjoint
bundle $\ad P:=P\times_\Ad\fg$ and defines the covariant differentials
$D_i=D^{\ad P}_i\colon\Om^i(M,\ad P)\to\Om^{i+1}(M,\ad P)$, where
$0\le i<\dim M$.
If $D$ is flat, then
\begin{equation}\label{eqn:cplx}
0\to\Om^0(M,\ad P)\stackrel{D_0}\longrightarrow\Om^1(M,\ad P)
\stackrel{D_1}\longrightarrow\cdots\stackrel{D_{n-1}}\longrightarrow
\Om^n(M,\ad P)\to0
\end{equation}
is an elliptic complex;
here $D_0$ is the infinitesimal action of $\cG(P)$ on $\cA(P)$.
Let $H^i(M,\ad P):=\ker D_i/\im D_{i-1}$ be the de Rham cohomology
groups with coefficients in $\ad P$.
Since $M$ is compact, the {\em Betti numbers}
$b_i(M,\ad P):=\dim_\bC H^i(M,\ad P)$ are finite.
Let $\chi(M,\ad P):=\sum_{i=0}^n(-1)^ib_i(M,\ad P)$ be the
{\em Euler characteristic} of the complex \eqref{eqn:cplx}.
It follows from the index formula \cite{AS3} that
$\chi(M,\ad P)=\chi(M)\dim_\bC G$
(see for example \cite[Eqn (4.117), p.117]{Na}).

The Lie algebra $\fg$ of a reductive Lie group $G$ has a decomposition
$\fg=Z(\fg)\oplus\,\fg'$, where $\fg':=[\fg,\fg]$.
Let $\Ad'$, $\ad'$ denote the restrictions to $\fg'$ of the adjoint
representations of $G$ and $\fg$, respectively.
Then there is also a decomposition $\ad P=Z(\fg)_M\oplus\ad'P$, where
$Z(\fg)_M$ is the trivial bundle over $M$ with fiber $Z(\fg)$ and
$\ad'P:=P\times_{\Ad'}\,\fg'$.
The decomposition is preserved by the connection $D$, which is trivial
on $Z(\fg)_M$ and will be denoted by $D'$ on $\ad'P$.
We have
\begin{align*}
\Om^i(M,\ad P)=\Om^i(M,Z(\fg))\oplus\Om^i(M,\ad'P),\quad
& H^i(M,\ad P)=H^i(M,Z(\fg))\oplus H^i(M,\ad'P), \\
b_i(M,\ad P)=b_i(M)\dim_\bC Z(\fg)+b_i(M,\ad'P),&\quad
\chi(M,\ad'P)=\chi(M)\dim_\bC(G/Z(G)).
\end{align*}

There is a $G$-invariant symmetric non-degenerate complex bilinear form
$(\cdot,\cdot)$ on $\fg$ such that the decomposition $\fg=Z(\fg)\oplus\,\fg'$
is orthogonal.
It defines a fiberwise complex bilinear form on $\ad P$ preserved by the
connection and such that the splitting $\ad P=Z(\fg)_M\oplus\ad'P$ is also
orthogonal.

So far, the compact manifold $M$ can be either orientable or non-orientable.
Now we assume that $M$ is orientable.
Using the above complex bilinear form of $\ad P$ and an orientation on $M$,
there is a non-degenerate complex bilinear pairing $\bra\cdot,\cdot\ket$ on
$\Om^\bullet(M,\ad P)$ given by
\begin{equation}\label{eqn:pairing}
\Om^i(M,\ad P)\times\Om^{n-i}(M,\ad P)\to\bC,\qquad
(\al,\beta)\mapsto\bra\al,\beta\ket:=\int_M(\al,\wedge\beta).
\end{equation}
This induces the following Poincar\'e duality on the de Rham cohomology
with coefficients in a flat bundle.

\begin{lem}\label{lem:dual}
If $M$ is orientable, there is a non-degenerate complex bilinear pairing,
still denoted by $\bra\cdot,\cdot\ket$,
\begin{equation}\label{eqn:Poincare}
H^i(M,\ad P)\times H^{n-i}(M,\ad P)\to\bC.
\end{equation}
Hence there is an isomorphism $H^i(M,\ad P)^*\cong H^{n-i}(M,\ad P)$,
and $b_i(M,\ad P)=b_{n-i}(M,\ad P)$.
The results remain true if the bundle $\ad P$ is replaced by $\ad'P$.
\end{lem}

\begin{proof}
Consider the dual complex of \eqref{eqn:cplx},
\[ 0\gets\Om^0(M,\ad P)'\stackrel{D_0^\T}\longleftarrow\Om^1(M,\ad P)'
\stackrel{D_1^\T}\longleftarrow\cdots\stackrel{D_{n-1}^\T}\longleftarrow
\Om^n(M,\ad P)'\gets0,  \]
where $\Om^i(M,\ad P)'$ is the space of $(n-i)$-currents on $M$ with
coefficients in $\ad P$ and $D_i^\T$ is the transpose of $D_i$.
Let $H_i(\Om^\bullet(M,\ad P)')=\ker D_{i-1}^\T/\im D_i^\T$.
Then \eqref{eqn:pairing} induces a non-degenerate pairing
\[ H^i(M,\ad P)\times H_i(\Om^\bullet(M,\ad P)')\to\bC. \]
The non-degenerate pairing \eqref{eqn:pairing} also defines an inclusion
$\Om^{n-i}(M,\ad P)\hookrightarrow\Om^i(M,\ad P)'$, and the restriction of the
map $D_{i-1}^\T$, whose domain is $\Om^i(M,\ad P)'$, to $\Om^{n-i}(M,\ad P)$ is
$(-1)^{i-1}D_{n-i}$.
By a slight generalisation (to allow coefficients in $\ad P$) of
\cite[Thm~14, p.79]{dR}, the above (co)homology groups defined using currents
can also be computed by smooth forms, i.e.,
$H_i(\Om^\bullet(M,\ad P)')\cong H^{n-i}(M,\ad P)$.
The results on $\ad P$ then follow.
Finally, the pairing \eqref{eqn:pairing} splits block-diagonally under the
decomposition $\Om^1(M,\ad P)=\Om^1(M,Z(\fg))\oplus\Om^i(M,\ad'P)$, and
the results  on $\ad'P$ also follow.
\end{proof}

\subsection{The slice theorem}
A flat connection $D\in\cA^\fla(P)$ is {\em simple} if the kernel of
$D'_0\colon\Om^0(M,\ad'P)\to\Om^1(M,\ad'P)$ is zero, or if $H^0(M,\ad'P)=0$.
If $D$ is a simple connection, then the Lie algebra of the stabiliser
$\cG(P)_D$ of $D$ is $\ker(D_0)=H^0(M,\ad P)\cong Z(\fg)$.
So the stabiliser has minimal dimension as $\ZG$ always acts trivially
on $\cA(P)$.
Conversely, if the stabiliser has the same dimension as $Z(G)$, then
$H^0(M,\ad'P)=0$ and the connection is simple.
We choose a Riemannian metric on $M$ and a Hermitian structure on $\ad P$
such that $\ad'P$ is orthogonal to $Z(\fg)_M$.
Then there are Hermitian inner products on $\Om^i(M,\ad P)$, $\Om^i(M,\ad'P)$
regardless of whether $M$ is orientable or not.
Let $D_i^\dagger$ be the formal adjoint of $D_i$.
For $D\in\cA^\fla(P)$, the {\em slice} at $D$ is the subset
\begin{equation}\label{eqn:slice}
\cS(D):=\big\{\al\in\Om^1(M,\ad P):
D_1\al+\tfrac12[\al,\al]=0,\,D_0^\dagger\al=0\big\}.
\end{equation}
The first goal is to prove a slice theorem for an arbitrary principal bundle
with a reductive structure group.
Let $\Om^\bullet_k(M,\ad P)$, $\cA^\fla_k(P)$, $\cG_k(P)$, $\cS_k(D)$ be the
respective spaces completed according to the Sobolev norm $\|\cdot\|_{2,k}$.

\begin{thm}\label{thm:slice}
Let $M$ be a compact manifold, $G$ be a complex reductive Lie group and
$P\to M$ be a $G$-bundle.
Fix $k>\frac{n}2$.
Suppose $D\in\cA^\fla_k(P)$ is simple.
Then there is a neighborhood $\cV$ of $0\in\cS_k(D)$ such that the map
\[\cG_{k+1}(P)/Z(G)\times\cV\to\cA^\fla_k(P),\qquad
  ([g],\al)\mapsto g\cdot(D+\al) \]
is a homeomorphism onto its image.
\end{thm}

\begin{proof}
We follow in part the proof of slice theorem in \cite{Ki,Ko}, though
the context there was on holomorphic structure on vector bundles.
Let $(\ker D_0)^\perp_{k+1}$ be the orthogonal complement of $\ker D_0$
in $\Om^0_{k+1}(M,\ad P)$.
Consider the map
\[ F\colon(\ker D_0)^\perp_{k+1}\times\Om^1_k(M,\ad P)\to\Om^0_k(M,\ad P),
   \qquad (u,\al)\mapsto D_0^\dagger(e^u\cdot(D+\al)-D).   \]
The differential of $F$ along $u$ is $-D_0^\dagger D_0$, which is an
invertible operator on $(\ker D_0)^\perp_{k+1}$.
By the implicit function theorem, if $\al$ is in a sufficiently small
neighbourhood $\cV$, there exists $u\in(\ker D_0)^\perp_{k+1}$ such that
$\beta:=e^u\cdot(D+\al)-D$ satisfies $D_0^\dagger\beta=0$.
Since $D$ is simple, there is a small neighbourhood $\cU$ of
$0\in(\ker D_0)^\perp_{k+1}$ such that the map
$u\in\cU\mapsto[e^u]\in\cG_{k+1}(P)/Z(G)$ is a diffeomorphism onto its image.
In this way, we obtain the desired homeomorphism provided we restrict to
small (close to identity) gauge transformations.
To get the full version, suppose $D+\al_1,D+\al_2$ are flat connections
such that $\al_1,\al_2$ are small.
We want to show that if $g\in\cG_{k+1}(P)$ such that $g\cdot(D+\al_1)=D+\al_2$,
then $[g]\in\cG_{k+1}(P)/Z(G)$ is also small, i.e., $g$ is close to $Z(G)$.
If $\fg'$ is a simple Lie algebra, then the simple connection $D$ on $P$
induces a simple connection on $\ad'P$, i.e., the kernel of
$D_0^{\End(\ad'P)}\colon\Om^0(M,\End(\ad'P))\to\Om^1(M,\End(\ad'P))$ is zero.
We apply the arguments in \cite{Ki,Ko} to the vector bundle $\ad'P$.
Since $D^{\ad'P}+\ad'(\al_1),D^{\ad'P}+\ad'(\al_2)$ are connections on $\ad'P$
related by the gauge transformation $\Ad'(g)$, we have
$\Ad'(g)=c(\id_{\ad'P}+g')$ for some constant $c\ne0$ and a small
$g'\in\Gam_{k+1}(\End(\ad'P))$.
More precisely, we have
\[ \|g'\|_{2,k+1}\le
   \frac{c_2\|\id_{\ad'P}\|_{2,k+1}(\|\al_1\|_{2,k}+\|\al_2\|_{2,k})}
   {c_1-c_2(\|\al_1\|_{2,k}+\|\al_2\|_{2,k})}, \]
where $c_1,c_2>0$ are constants such that
$\|[D^{\ad'P},g']\|_{2,k}\ge c_1\|g'\|_{2,k+1}$ and
$\|g'\al\|_{2,k}\le c_2\|g'\|_{2,k+1}\|\al\|_{2,k}$, respectively.
We want to show that $c$ is close to $1$, and hence $\Ad'(g)$ is close to
$\id_{\ad'P}$, or $g$ is close to $Z(G)$.
Let $c_3>0$ be a constant such that
$\|[\xi,\eta]\|_{2,k}\le c_3\|\xi\|_{2,k}\|\eta\|_{2,k}$ for all
$\xi,\eta\in\Om^0_k(M,\ad'P)$.
Since $\fg'$ is non-Abelian, we can fix $\xi,\eta$ such that
$\zeta:=[\xi,\eta]\in\Om^0(M,\ad'P)$ is non-zero.
Let $c_4=\|\zeta\|_{2,k}$, $c_5=c_2c_3\|\xi\|_{2,k}\|\eta\|_{2,k}$ and
$c_6=c_2c_5$ be positive constants.
Then we have
\[ \|[\xi,g'\eta]\|_{2,k}\mbox{ and }
\|[g'\xi,\eta]\|_{2,k}\le c_5\|g'\|_{2,k+1},
\qquad \|[g'\xi,g'\eta]\|_{2,k}\le c_6\|g'\|_{2,k+1}^2. \]
Finally $\Ad'(g)\zeta=[\Ad'(g)\xi,\Ad'(g)\eta]$, or
$\zeta+g'\zeta=c\,[\xi+g'\xi,\eta+g'\eta]$ implies
\[ \frac{c_4(1-c_2\|g'\|_{2,k+1})}
{c_4+2c_5\|g'\|_{2,k+1}+c_6\|g'\|_{2,k+1}^2} \le c\le
\frac{c_4(1+c_2\|g'\|_{2,k+1})}
{c_4-2c_5\|g'\|_{2,k+1}-c_6\|g'\|_{2,k+1}^2}. \]
When $\|\al_1\|_{2,k}$, $\|\al_2\|_{2,k}$ are sufficiently small, so
is $\|g'\|_{2,k+1}$, and thus $c$ is sufficiently close to $1$.
Therefore $\Ad'(g)$ is close to the identity.
In general, $\fg'$ is semisimple and suppose $\fg'=\oplus_{i=1}^r\fg_i$
is a decomposition into simple Lie algebras.
Let $\Ad'_i$ be the restriction of the adjoint representation of $G$ to
$\fg_i$.
Then there is a decomposition of vector bundles $\ad'P=\oplus_{i=1}^r\ad'_iP$,
where $\ad'_iP:=P\times_{\ad'_i}\fg_i$, which is respected by both the gauge
transformation $\Ad'(g)$ and the induced connection on $\ad'P$.
The above argument shows that on each $\ad'_iP$, $\Ad'_i(g)$ is close to the
identity.
Therefore so is the whole $\Ad'(g)$ and hence the result.
\end{proof}

In the above proof, if $\al\in\Om^1(M,\ad P)$ is smooth, then so is the
solution $u\in(\ker D_0)_{k+1}^\perp$ to the elliptic equation $F(u,\al)=0$.
Similarly, if $\al_1,\al_2\in\Om^1(M,\ad P)$ are smooth, then so is the
solution $g\in\cG(P)$ to the elliptic equation
$g^{-1}Dg=\al_1-\Ad(g^{-1})\al_2$.
Therefore the result of Theorem~\ref{thm:slice} remains true if the Sobolev
spaces are replaced by spaces $\Om^\bullet(M,\ad P)$, $\cA(P)$, $\cG(P)$,
$\cS(D)$ of smooth objects.

A flat connection $D$ is {\em reductive} if the closure of the holonomy group
$\Hol(D)$ is contained in the Levi subgroup of any parabolic subgroup
containing $\Hol(D)$.
Let $\cA^{\fla,\red}(P)$ be the set of flat, reductive connections on $P$.
The {\em moduli space of flat connections} on $P$ is
$\cM^\fla(P):=\cA^{\fla,\red}(P)/\cG(P)$.
A flat connection $D$ is reductive if and only if its orbit under the group
$\cG(P)$ of gauge transformations is closed (see \S\ref{sec:compare}).
Thus the quotient topology on the moduli space $\cM^\fla(P)$ is Hausdorff.
If a flat connection $D$ is simple, then its stabiliser $\cG(P)_D$ contains
$Z(G)$ as a subgroup of equal dimension.
A flat connection $D$ is {\em good} (cf.~\cite{JM}) if it is reductive and
its stabiliser $\cG(P)_D$ is $Z(G)$.
The slice at a good connection gives a local model for the moduli space.

\begin{cor}\label{cor:slice}
Let $D$ be a reductive, simple flat connection on $P$.
Then the map
\[ p\colon \cS(D)/\cG(P)_D\to\cA^\fla(P)/\cG(P),\qquad[\al]\mapsto[D+\al] \]
is a homeomorphism of a neighbourhood of $[0]$ in $\cS(D)/\cG(P)_D$ onto a
neighbourhood of $[D]$ in $\cM^\fla(P)$.
If in addition $D$ is good, then the homeomorphism is from a neighbourhood
of $0$ in $\cS(D)$ onto a neighbourhood of $[D]$ in $\cM^\fla(P)$.
\end{cor}

\begin{proof}
By Theorem~\ref{thm:slice}, $p$ is well defined, continuous and one-to-one
near $[0]$.
For a sufficiently small neighbourhood $\cU$ of $0$ in $\cS(D)$ such that the
connection $D+\al$ is reductive and simple for any $\al\in\cU$, the preimage
of $p(\cU/\cG(P))\subset\cM^\fla(P)$ in $\cA^\fla(P)$ under the projection is
$\cG(P)\cdot\cU$.
This is an open subset and, since the orbit $\cG(P)\cdot D$ is closed, it
descends to an open neighbourhood of $[D]$ in $\cM^\fla(P)$.
So $p^{-1}$ is also continuous.
\end{proof}

\subsection{Smooth points on the moduli space}\label{sec:smooth-g}
A separate issue is the smoothness of $\cA^\fla(P)$ itself near a flat
connection $D$.
We will show that $H^2(M,\ad'P)$ is the {\em obstruction} to smoothness
from the gauge theoretic point of view.
The proof outlined below is similar to the proofs in \cite{Ki,Ko} for moduli
of holomorphic bundles and of Hermitian-Einstein connections on K\"ahler
manifolds.
However we are studying moduli of flat $G$-connections on a Riemannian
manifold.
With the Riemannian metric on $M$ and the Hermitian structure chosen above, the
Laplacian $\Box_i:=D_i^\dagger D_i+D_{i-1}D_{i-1}^\dagger$ and the associated
Green's operator $G_i$ preserve the decomposition
$\Om^i(M,\ad P)\cong\Om^i(M,Z(\fg))\oplus\Om^i(M,\ad'P)$ for $0\le i\le n$.
In particular, a harmonic form in $\Om^i(M,\ad P)$ projects to harmonic
forms in $\Om^i(M)\otimes Z(\fg)$ and $\Om^i(M,\ad'P)$.
We define the {\em Kuranishi map}
\[  \kappa\colon\Om^1(M,\ad P)\to\Om^1(M,\ad P),\qquad
    \al\mapsto\al+\tfrac12D_1^\dagger G_2[\al,\al].   \]

\begin{pro}\label{pro:Kur}
Let $M$ be a compact manifold and $G$ be a complex reductive Lie group.
Suppose $D$ is a flat connection on a $G$-bundle $P\to M$ such that
$H^2(M,\ad'P)=0$.
Then for $k>\frac{n}2+2$, there is a neighbourhood of $D$ in $\cA^\fla_k(P)$
and a neighbourhood of $0$ in $\ker D_1\subset\Om^1_k(M,\ad P) $ that are
diffeomorphic via the map $D+\al\mapsto\kappa(\al)$.
\end{pro}

\begin{proof}
As the differential of $\kappa$ is the identity map at $0$, it is invertible
near $0$.
For $\al\in\Om^1(M,\ad P)$, we have $[\al,\al]\in\Om^2(M,\ad'P)$.
By the assumption $H^2(M,\ad'P)=0$, the Green's operator $G_2$ is the inverse
of the Laplacian $\Box_2=D_2^\dagger D_2+D_1D_1^\dagger$ on $\Om^2(M,\ad'P)$.
Therefore $\Box_2G_2[\al,\al]=G_2\Box_2[\al,\al]=[\al,\al]$ and
\[ D_1\kappa(\al)=D_1\al+\tfrac12D_1D_1^\dagger G_2[\al,\al]
   =D_1\al+\tfrac12[\al,\al]-\tfrac12G_2D_2^\dagger D_2[\al,\al]. \]
Here since $\Box_2$ commutes with $D_1D_1^\dagger$ and $D_2^\dagger D_2$,
so does $G_2$.
If $D_1\al+\tfrac12[\al,\al]=0$, then $-\tfrac12D_2[\al,\al]=D_2D_1\al=0$,
and so $D_1\kappa(\al)=0$.
Conversely, if $D_1\kappa(\al)=0$, we want to show that
$\gam:=D_1\al+\tfrac12[\al,\al]$ is zero.
Using
\[ \gam=\tfrac12G_2D_2^\dagger D_2[\al,\al]=G_2D_2^\dagger[D_1\al,\al]
=G_2D_2^\dagger[\gam,\al], \]
we get, for some constant $c_0>0$,
\[ \|\gam\|_{2,k-1}\le c_0\,\|\al\|_{2,k-2}\,\|\gam\|_{2,k-2}
   \le c_0\,\|\al\|_{2,k}\,\|\gam\|_{2,k-1}.\]
Therefore $\gam=0$ when $\|\al\|_{2,k}$ is sufficiently small.
\end{proof}

{}From the proof of Proposition~\ref{pro:Kur}, it is evident that the
result holds if we assume, instead of $H^2(M,\ad'P)=0$, that the harmonic
part of $[\al,\al]$ is zero for all $\al\in\Om^1(M,\ad P)$.
The condition itself depends on a metric on $M$, but it implies that the map
$H^1(M,\ad P)\to H^2(M,\ad'P)$, sending the cohomology class represented by
a closed $1$-form $\al\in H^1(M,\ad P)$ to the cohomology class of
$[\al,\al]$, is zero.

As in the slice theorem, Proposition~\ref{pro:Kur} remains valid if we
restrict to the spaces of smooth objects.
Let $\cA^\fla_\circ(P)$ be the set of reductive flat connections that
are good and satisfy the condition $H^2(M,\ad'P)=0$, and let
$\cM^\fla_\circ(P):=\cA^\fla_\circ(P)/\cG(P)$.
Combining Proposition~\ref{pro:Kur} and Corollary~\ref{cor:slice}, we
conclude that $\cM^\fla_\circ(P)$ is in the smooth part of the moduli
space $\cM^\fla(P)$.
However, there can be smooth points outside $\cM^\fla_\circ(P)$.
This can be seen from the examples of character varieties in the Appendix
(cf.~\S\ref{sec:luna}), as $\cM^\fla_\circ(P)$ coincides with its counterpart
in the character variety when $M$ is a surface (cf.~\S\ref{sec:compare}).

\begin{cor}\label{cor:mod}
Suppose $D\in\cA^\fla_\circ(P)$.
Then there is a neighbourhood of $[D]$ in $\cM^\fla_\circ(P)$ diffeomorphic
to a neighbourhood of $0$ in $H^1(M,\ad P)$.
In particular, $\dim_\bC\cM^\fla_\circ(P)=b_1(M,\ad P)$ is finite.
\end{cor}

When $M$ is a compact orientable surface (i.e., $n=2$), the two cohomology
groups $H^0(M,\ad'P)$ and $H^2(M,\ad'P)$ are dual spaces of each other by
Lemma~\ref{lem:dual}.
So a good flat connection represents a smooth point in the moduli space.
Moreover, the dimension of $\cM^\fla_\circ(P)$ can be computed by an index
formula.

\begin{cor}\label{cor:surf}
Suppose $M$ is a compact orientable surface of genus $g>1$ and
$P\to M$ is a flat principal $G$-bundle.
Then for any good flat $D\in\cA^\fla(P)$, we have $H^2(M,\ad'P)=0$, and
there is a neighbourhood of $[D]$ in $\cM^\fla_\circ(P)$ diffeomorphic to
a neighbourhood of $0$ in $H^1(M,\ad P)$.
Moreover,
\[  \dim_\bC\cM^\fla_\circ(P)=(2g-2)\dim_\bC G+2\dim_\bC Z(\fg). \]
\end{cor}

\begin{proof}
By Corollary~\ref{cor:mod}, the dimension of the moduli space is
\begin{eqnarray*}
&b_1(M,\ad P)=b_1(M,\ad'P)+b_1(M)\dim_\bC Z(\fg)
 =-\chi(M,\ad'P)+2g\dim_\bC Z(\fg) \\
&\quad=(2g-2)\dim_\bC\fg'+2g\dim_\bC Z(\fg)
=(2g-2)\dim_\bC\fg+2\dim_\bC Z(\fg).
\end{eqnarray*}
\vspace{-1cm}

\end{proof}

In this case ($n=2$), the pairing \eqref{eqn:Poincare} is a (complex)
symplectic structure on $H^1(M,\ad P)$.
In fact, the symplectic reduction procedure of \cite{AB} yields a holomorphic
symplectic form $\upom$ on $\cM^\fla_\circ(P)$ which restricts to the above
one on the tangent space at each $[D]\in\cM^\fla_\circ(P)$.
More generally, if $\dim M=n$ is even and $M$ is a compact K\"ahler manifold
with a K\"ahler form $\om$, then a symplectic form $\upom$ on
$\cM^\fla_\circ(P)$, or $\upom_{[D]}$ on $H^1(M,\ad P)$, is defined by
\begin{equation}\label{eqn:omD}
\upom_{[D]}([\al],[\beta])=\int_M(\al,\wedge\beta)\wedge
   \frac{\om^{\frac{n}2-1}}{\big(\frac{n}2-1\big)!},
\end{equation}
where $\al,\beta\in\Om^1(M,\ad P)$ are closed.

\subsection{Non-orientable manifolds}\label{sec:non-g}
Now suppose $M$ is a compact, non-orientable manifold of dimension $n$.
Let $\pi\colon\tM\to M$ be the orientable double cover.
The non-trivial deck transformation $\tau\colon\tM\to\tM$ is an involution
on $\tM$.
Given a principal bundle $P\to M$ with complex reductive structure group $G$,
let $\tP:=\pi^*P$ denote the pullback to $\tM$.
Then there is a lift of the involution $\tau$ to $\tP$, still denoted by
$\tau$, such that $\tP/\tau=P$, and the map $\pi^*$ pulls back forms,
connections and gauge transformation from $M$ to $\tM$.
The pullback map $\tau^*$ acts as an involution on $\Om^i(\tM,\ad\tP)$,
$\cA(\tP)$, $\cG(\tP)$, and the $\tau^*$-invariant subspaces can be identified
with the corresponding spaces from the bundle $P\to M$ \cite{HL}, i.e.,
we have the following isomorphisms via $\pi^*$:
\[  \Om^i(M,\ad P)\cong\Om^i(\tM,\ad\tP)^\tau,\quad
\cA(P)\cong\cA(\tP)^\tau,\quad\cG(P)\cong\cG(\tP)^\tau. \]
Suppose a connection $D$ on $P$ pulls back to $\tD$ on $\tP$.
Then $\tD$ is flat if and only if $D$ is so, and
$\cA^\fla(P)\cong\cA^\fla(\tP)^\tau$.
The covariant differentials $D_i\colon\Om^i(M,\ad P)\to\Om^{i+1}(M,\ad P)$ and
$\tD_i\colon\Om^i(\tM,\ad\tP)\to\Om^{i+1}(\tM,\ad\tP)$ satisfy
\[ \pi^*\circ D_i=\tD_i\circ\pi^*,\qquad\tau^*\circ\tD_i=\tD_i\circ\tau^*.\]
Therefore $\tau^*$ acts as an involution on $H^i(\tM,\ad\tP)$ for
$0\le i\le n$.
Let
\[  H^i(\tM,\ad\tP)=H^i(\tM,\ad\tP)^\tau\oplus H^i(\tM,\ad\tP)^{-\tau}  \]
be the decomposition such that $\tau^*=\pm1$ on the subspaces
$H^i(\tM,\ad\tP)^{\pm\tau}$, respectively.
Set $b_i^\pm(\tM,\ad\tP):=\dim_\bC H^i(\tM,\ad\tP)^{\pm\tau}$.
We have similar decompositions for $H^i(\tM,\bC)$, $H^i(\tM,\ad'\tP)$,
and $b_i^\pm(\tM)=\dim_\bC H^i(\tM,\bC)^{\pm\tau}$,
$b_i^\pm(\tM,\ad'\tP)=\dim_\bC H^i(\tM,\ad'\tP)^{\pm\tau}$.

\begin{lem}\label{lem:tau}
Suppose a flat connection $D$ on $P\to M$ and it lifts to a flat connection
$\tD$ on $\tP\to\tM$.\\
1. There are non-degenerate bilinear pairings
\begin{equation}\label{eqn:pairing+-}
H^i(\tM,\ad\tP)^{\pm\tau}\times H^{n-i}(\tM,\ad\tP)^{\mp\tau}\to\bC.
\end{equation}
Hence $(H^i(\tM,\ad\tP)^{\pm\tau})^*\cong H^{n-i}(\tM,\ad\tP)^{\mp\tau}$,
$b_i^\pm(\tM,\ad\tP)=b_{n-i}^\mp(\tM,\ad\tP)$.\\
2. There are isomorphisms $H^i(\tM,\ad\tP)^\tau\cong H^i(M,\ad P)$,
$H^i(\tM,\ad\tP)^{-\tau}\cong H^{n-i}(M,\ad P)^*$.
Hence
\[ b_i^+(\tM,\ad\tP)=b_i(M,\ad P),\quad b_i^-(\tM,\ad\tP)=b_{n-i}(M,\ad P),
   \quad b_i(\tM,\ad\tP)=b_i(M,\ad P)+b_{n-i}(M,\ad P).  \]
The same results hold for $H^i(\tM,\bC)$ and $H^i(\tM,\ad'\tP)$.
\end{lem}

\begin{proof}
1. Since $\tau$ reverses the orientation on $\tM$, we have, for all
$[\al]\in H^i(\tM,\ad\tP)$, $[\beta]\in H^{n-i}(\tM,\ad\tP)$,
\[ \bra\tau^*[\al],\tau^*[\beta]\ket=-\bra[\al],[\beta]\ket.  \]
Therefore, the non-degenerate pairing $\bra\cdot,\cdot\ket$ for
$H^\bullet(\tM,\ad\tP)$ splits into two in \eqref{eqn:pairing+-}.\\
2. The first isomorphism is because of the isomorphism
$\pi^*\colon\Om^\bullet(M,\ad P)\to\Om^\bullet(\tM,\ad\tP)^\tau$ of
cochain complexes.
Then $H^i(\tM,\ad\tP)^{-\tau}\cong(H^{n-i}(\tM,\ad\tP)^\tau)^*\cong
H^{n-i}(M,\ad P)^*$.
The rest follows easily.
\end{proof}

Since $\chi(\tM,\ad\tP)=\chi(\tM)\dim_\bC G$, $\chi(M,\ad P)=\chi(M)\dim_\bC G$
and $\chi(M)=\frac12\chi(\tM)$, we get $\chi(M,\ad P)=\frac12\chi(\tM,\ad\tP)$;
both sides vanish if $\dim M=n$ is odd.
On the other hand, the Lefschetz number of $\tau$ is the supertrace of $\tau^*$
on $H^\bullet(\tM,\ad\tP)$, i.e.,
\[ L(\tau,\ad\tP):=\sum_{i=0}^n(-1)^i\tr(\tau^*|H^i(\tM,\ad\tP)). \]
In our case, since $\tau$ acts on $\tM$ without fixed points, we get
$L(\tau,\ad\tP)=0$ regardless of whether $n$ is even or odd.
These statements are consistent with Lemma~\ref{lem:tau}.

It can be shown that a flat connection $D$ is reductive if and only if the
pullback $\tD$ is so \cite{HWW}.
On the other hand, $D$ is simple if $b_0(M,\ad'P)=0$, whereas $\tD$ is simple
if $b_0(\tM,\ad\tP)=0$.
Clearly, $D$ is simple if $\tD$ is so, but the converse is not true
(Example~\ref{eg:simple}).
Similarly, $D$ is good if $\tD$ is so, but the converse is not true either
(Example~\ref{eg:nonCI}).
In addition to the requirement in the orientable case that the flat connection
is good, the smoothness of $\cM^\fla(P)$ at $[D]$ requires further that
$H^2(M,\ad'P)=0$, whereas that of $\cM^\fla(\tP)$ at $[\tD]$ requires
$H^2(\tM,\ad'\tP)=0$.
By Lemma~\ref{lem:tau}, the vanishing of $H^2(\tM,\ad'\tP)$ implies that of
$H^2(M,\ad'P)$, but the converse is not true (Example~\ref{eg:h2}).
We refer the reader to the Appendix for various examples.

\begin{pro}\label{pro:non-r}
Suppose a flat connection $D$ on $P\to M$ lifts to a good flat connection
$\tD$ on $\tP\to\tM$.
Then\\
1. $\pi^*\colon\cM^\fla(P)\to\cM^\fla(\tP)^\tau$ is a homeomorphism from a
neighbourhood of $[D]$ to a neighbourhood of $[\tD]$.\\
2. if furthermore $H^2(\tM,\ad'P)=0$, the above local homeomorphism is a local
diffeomorphism, and
$\dim_\bC\cM^\fla_\circ(P)=b_1^+(\tM,\ad\tP)=b_{n-1}^-(\tM,\ad\tP)$.
\end{pro}

\begin{proof}
1. There is an induced $\tau$-action on $\cS(\tD)$ such that $\cS(\tD)^\tau$
can be identified (via $\pi^*$) with $\cS(D)$.
The map $\cS(\tD)\to\cA^\fla(\tP)/\cG(\tP)$ in Corollary~\ref{cor:slice} is
$\tau$-equivariant.
Choose a sufficiently small $\tau$-invariant neighbourhood $\tilde\cV$ of
$0\in\cS(\tD)$.
Then $\tilde\cV^\tau$ is a neighbourhood of $0\in\cS(\tP)^\tau$ that is
homeomorphic to a neighbourhood of $[\tD]\in\cM^\fla(P)^\tau$.
On the other hand, $\tilde\cV^\tau$ can be identified (via $\pi^*$) with a
neighbourhood $\cV$ of $0\in\cS(D)$ and is homeomorphic to a neighbourhood of
$[D]\in\cM^\fla(P)$.\\
2. This follows from Corollary~\ref{cor:mod} and Lemma~2.7.2.
\end{proof}

\begin{cor}\label{cor:nbhd}
Suppose $P$ is a $G$-bundle over a compact non-orientable surface $M$
homeomorphic to the connected sum of $h>2$ copies of $\bR\mathrm{P}^2$ and
suppose $D$ is a flat connection on $P$ whose pullback $\tD$ to $\tP\to\tM$
is good.
Then there is a neighbourhood of $[D]$ in $\cM^\fla_\circ(P)$ diffeomorphic
to a neighbourhood of $\tD$ in $\cM^\fla_\circ(\tP)^\tau$.
Moreover,
\[  \dim_\bC\cM^\fla_\circ(P)=\tfrac12\dim_\bC\cM^\fla_\circ(\tP)=
    (h-2)\dim_\bC G+\dim_\bC Z(G).    \]
\end{cor}

\begin{proof}
Since $\tM$ is an orientable surface of genus $h-1$, we have, by
Lemma~\ref{lem:tau},
\[  \dim_\bC H^1(\tM,\ad'\tP)^\tau=\dim_\bC H^1(\tM,\ad'\tP)^{-\tau}
    =\tfrac12\dim_\bC H^1(\tM,\ad'\tP).   \]
The results then follow easily from Proposition~\ref{pro:non-r}.2 and
Corollary~\ref{cor:surf}.
\end{proof}

Finally, if $\tM$ is K\"ahler, then the action of $\tau$ on
$\cM^\fla_\circ(\tP)$ is anti-symplectic with respect to $\upom$
in \eqref{eqn:omD}.
Therefore, $\cM^\fla_\circ(\tP)^\tau$ is an isotropic submanifold in
$\cM^\fla_\circ(\tP)$.
If $\dim M=2$, then $\cM^\fla_\circ(\tP)^\tau$ is a Lagrangian submanifold
\cite{HL}.

\section{Algebraic approach to character varieties}\label{sec:rep}
\subsection{Smooth points on the homomorphism space}\label{sec:smooth-r}
We assume that $\PI$ is a finitely generated group, i.e., $\PI=F/N$, where
$F=\bra X\ket$ is the free group on a finite set $X:=\{x_1,\dots,x_d\}$ and
$N$ is a normal subgroup in $F$.
An element $w\in F$ is a word in $X$, i.e., $w=\prod_{k=1}^lx_{i_k}^{m_k}$,
where $m_k\in\bZ\backslash\{0\}$ ($k=1,\dots,l$).
We also assume that $\PI$ is finitely presented, that is, in addition, $N$ is
the normal closure in $F$ of a finite set $R:=\{r_1,\dots,r_q\}\subset N$.
Each $r_j$ ($j=1,\dots,q$) is called a {\em relator}, and an element of $N$
is of the form $\prod_{k=1}^ms_kr_{j_k}^{n_k}s_k^{-1}$, where $s_k\in F$,
$1\le j_k\le q$, $n_k\in\bZ\backslash\{0\}$ ($k=1,\dots,m$).
Given a connected Lie group $G$, we have $\Hom(F,G)=G^X$ (the set of maps from
$X$ to $G$) since any homomorphism $\phi\colon F\to G$ is determined by its
values on the generators, $(\phi(x_i))_{i=1,\dots,d}\in G^X$.
Each word $w=\prod_{k=1}^lx_{i_k}^{m_k}\in F$ defines a map
$\tilde w\colon G^X\to G$,
$(g_i)_{i=1,\dots,d}\mapsto\prod_{k=1}^lg_{i_k}^{m_k}$.
In particular, we have the maps $\tilde r_j\colon G^X\to G$ ($j=1,\dots,q$)
from the relators, and they form a single map
$\br=(\tilde r_j)_{j=1,\dots,q}\colon G^X\to G^R$.
The space $\Hom(\PI,G)$ can be identified with the subset
$\br^{-1}(e,\dots,e)=\bigcap_{j=1}^q\tilde r_j^{-1}(e)$ of $G^X$.
We want to find a sufficient condition on $\phi\in\Hom(\PI,G)$ so that
the space $\Hom(\PI,G)$ is smooth at $\phi$.

Let $\fg$ be the Lie algebra of $G$.
Composition of $\phi\in\Hom(F,G)$ with the adjoint representation of $G$ on
$\fg$ makes $\fg$ a $\bZ F$-module, which we denote by $\fg_\Adp$.
Recall that a $1$-cocycles on $F$ with coefficients in $\fg_\Adp$ is a map
$\gam\colon F\to\fg$ such that $\gam(uv)=\gam(u)+\Ad_{\phi(u)}\gam(v)$
for all $u,v\in F$.
The space $Z^1(F,\fg_\Adp)$ of these $1$-cocycles can be identified with
$\fg^X$ as each $1$-cocycle is determined by its values on the generators
$x_1,\dots,x_d$.
On the other hand, by the left multiplication of $G$ on $G$, the tangent space
of $\Hom(F,G)=G^X$ at any point $\phi$ is also identified with $\fg^X$.
If $\gam=(\gam_i)_{i=1,\dots,d}\in\fg^X=T_\phi\Hom(F,G)$, let
$\tgam\in Z^1(F,\fg_\Adp)$ be the corresponding $1$-cocycle
satisfying $\tgam(x_i)=\gam_i$ ($i=1,\dots,d$).
Then for any $\phi\in\Hom(F,G)$ and $w\in F$, $\gam\in\fg^X$, we have \cite{We}
(see also \cite[\S VI]{Ra})
\begin{equation}\label{eqn:diff}
(d\tilde w)_\phi(\gam)=\tgam(w).
\end{equation}

There is an action of $\PI$ on $\Hom(N,\fg_\Adp)$ because $F$ acts on $N$ by
conjugation and on $\fg$ by $\Adp$, whereas $N$ acts on $\Hom(N,\fg_\Adp)$
trivially.
We denote the invariant subspace by $\Hom(N,\fg_\Adp)^\PI$.
There is an evaluation map $\ev_R\colon\Hom(N,\fg_\Adp)^\PI\to\fg^R$,
$\beta\mapsto(\beta(r_j))_{j=1,\dots,q}$.
The map $\ev_R$ is injective because if $\beta(R)=0$, then $\beta(N)=0$ by
$\PI$-invariance.
So we can regard $\Hom(N,\fg_\Adp)^\PI$ as the subspace $\im(\ev_R)$ in
$\fg^R$.

\begin{thm}\label{thm:rank}
Let $(d\br)_\phi\colon\fg^X\to\fg^R$ be the differential of the map
$\br\colon G^X\to G^R$ at $\phi\in\Hom(\PI,G)$.
Then $\ker(d\br)_\phi=Z^1(\PI,\fg_\Adp)$ and $\im(d\br)_\phi$ is contained in
the subspace $\im(\ev_R)$ of $\fg^R$ that is isomorphic to
$\Hom(N,\fg_\Adp)^\PI$, and
\begin{equation}\label{eqn:rank}
\rank(d\br)_\phi=\dim\Hom(N,\fg_\Adp)^\PI-\dim H^2(\PI,\fg_\Adp).
\end{equation}
If $\Hom(N,\fg_\Adp)^\PI\cong\fg^R$, then
$\coker(d\br)_\phi\cong H^2(\PI,\fg_\Adp)$.
\end{thm}

\begin{proof}
Identifying $\Hom(\PI,G)$ as a subset in $G^X$ as above, the expected tangent
space of $\Hom(\PI,G)$ at $\phi\in\Hom(\PI,G)$ is $\ker(d\br)_\phi=
\bigcap_{j=1}^q\ker(d\tilde r_j)_\phi\subset\fg^X=Z^1(F,\fg_\Adp)$.
By \eqref{eqn:diff}, we obtain $\ker(d\br)_\phi=\ker(\res^F_N)$, where
$\res^F_N\colon Z^1(F,\fg_\Adp)\to Z^1(N,\fg_\Adp)$ is the restriction map.
Therefore $\ker(d\br)_\phi=Z^1(\PI,\fg_\Adp)$.
Since $N$ acts on $\fg_\Adp=\fg$ trivially, we have
$Z^1(N,\fg_\Adp)=\Hom(N,\fg_\Adp)$.
In fact, the image of $\res^F_N$ is contained in the invariant subspace
$\Hom(N,\fg_\Adp)^\PI$.
So we write $\res^F_N\colon Z^1(F,\fg_{\Ad\circ\rho})\to\Hom(N,\fg_\Adp)^\PI$.
By \eqref{eqn:diff} again, we get a factorisation
$(d\br)_\phi=\ev_R\circ\res^F_N$ of the map $(d\br)_\phi$
through $\Hom(N,\fg_\Adp)^\PI$, that is, the triangle in the diagramme
\begin{equation}\label{eqn:diag}
\xymatrix@=1.5pc{
\cdots\ar[r] & Z^1(F,\fg_\Adp)\ar[r]^{\hspace{-1em}\res^F_N}
\ar[dr]_{(d\br)_\phi} & \Hom(N,\fg_\Adp)^\PI\ar[d]^{\ev_R}\ar[r]^\Del &
H^2(\PI,\fg_\Adp)\ar[r] & 0\\
&& \fg^R & & }
\end{equation}
is commutative.
Therefore we obtain
\[  \rank(d\br)_\phi=\rank(\res^F_N)=
    \dim\Hom(N,\fg_\Adp)^\PI-\dim\coker(\res^F_N).  \]
If $\ev_R$ is surjective (hence an isomorphism), then
$\coker(d\br)_\phi\cong\coker(\res^F_N)$.

By a classical result of Eilenberg-MacLane \cite{EM}, the horizontal row in
\eqref{eqn:diag} is an exact sequence.
There, $\Del(\beta)$ is the obstruction to lifting
$\beta\in\Hom(N,\fg_\Adp)^\PI$ to a $1$-cocycle on $F$.
More explicitly, let $c\colon\PI\to F$ be any (set-theoretic) section and let
$\tilde c\colon\PI\times\PI\to N$ be defined by
$c(a_1)c(a_2)=c(a_1a_2)\tilde c(a_1,a_2)$, where $a_1,a_2\in\PI$.
Then $\beta\circ\tilde c\in Z^2(\PI,\fg_\Adp)$ and its class
$\Del(\beta):=[\beta\circ\tilde c]\in H^2(\PI,\fg_\Adp)$ does not depend on
the choice of $c$.
This exact sequence also follows from the Lyndon-Hochschild-Serre spectral
sequence \cite{L48,HS} of group cohomology associated to $\PI=F/N$ or from
Gruenberg's resolution \cite{Gru}.
Consequently, $\coker(\res^F_N)\cong H^2(\PI,\fg_\Adp)$ and the results follow.
\end{proof}

The vanishing of $H^2(\PI,\fg_\Adp)$ is the condition that any vector in
$Z^1(\PI,\fg_\Adp)$ can be integrated to a curve in $\Hom(\PI,G)$
(cf.~\cite{G,GM}).
In \cite{G}, it was stated that the vanishing of $H^2(\PI,\fg_\Adp)$ is a
sufficient condition for $\phi$ to be a smooth point on $\Hom(\PI,G)$ when
$\PI$ is the fundamental group of a closed orientable surface.
We will confirm this in \S\ref{sec:fox}.
Using the implicit function theorem, $\phi$ is a smooth point on
$\Hom(\PI,G)$ if $\rank(d\br)_\phi$ reaches its maximum value.
Our result \eqref{eqn:rank} in Theorem~\ref{thm:rank} is that
$\rank(d\br)_\phi$ depends on both $H^2(\PI,\fg_\Adp)$ and
$\Hom(N,\fg_\Adp)^\PI\cong\im(\ev_R)$.
When the map $\ev_R$ is surjective, this reduces to the condition that
$\dim H^2(\PI,\fg_\Adp)$ is minimal, which is clearly satisfied if
$H^2(\PI,\fg_\Adp)=0$ or $H^2(\PI,\fg'_{\Ad'\circ\phi})=0$.
The surjectivity holds when $\PI$ has a presentation with a single relator
(Proposition~\ref{pro:one}), and in particular if $\PI$ is the fundamental
group of a closed orientable or non-orientable surface.
However, for an arbitrary closed manifold $M$ or for an arbitrary finitely
presented group $\PI$, the smoothness of $\Hom(\PI,G)$ at $\phi$ depends not
only on $H^2(\PI,\fg_\Adp)$ but also on the higher cohomology groups or the
higher terms in the resolution \eqref{eqn:resol}.
Unfortunately, we can not provide an example in which the smoothness is
actually affected by these higher terms, nor are we able to show in the
general case that the vanishing of $H^2(\PI,\fg_\Adp)$ or
$H^2(\PI,\fg'_{\Ad'\circ\phi})$ alone is sufficient for the smoothness of
$\Hom(\PI,G)$ at $\phi$ using the deformation techniques in \cite{NR1,NR2}.

\begin{cor}\label{cor:sm-hom}
Let $\PI=F/N$ be a finitely presented group and $G$ be a connected Lie group.
If at $\phi\in\Hom(\PI,G)$, the number
$\dim\Hom(N,\fg_\Adp)^\PI-\dim H^2(\PI,\fg_\Adp)$ reaches its maximal value,
then $\phi$ is a smooth point on $\Hom(\PI,G)$ and
$T_\phi\Hom(\PI,G)\cong Z^1(\PI,\fg_\Adp)$.
The smooth part of $\Hom(\PI,G)$ is of dimension
\[ |X|\dim G-\dim\Hom(N,\fg_\Adp)^\PI+\dim H^2(\PI,\fg_\Adp).  \]
\end{cor}

If $\ev_R$ is surjective or equivalently, if $\Hom(N,\fg_\Adp)^\PI\cong\fg^R$,
then the dimension of the smooth part of $\Hom(\PI,G)$ becomes
$(|X|-|R|)\dim G+\dim H^2(\PI,\fg_\Adp)$, where $\phi$ is chosen so that
$H^2(\PI,\fg_\Adp)$ is of minimal dimension.
In general, $\ev_R$ is not surjective as the relators in the presentation
might not be independent themselves.
A {\em relation among the relators} is of the form $s(r_1,\dots,r_q)=e$, where
$s$ is a word on $R$ such that it becomes the identity element when each $r_j$
($j=1,\dots,q$) is substituted by the word on $X$ it represents.
More precisely, let $F_R$ be the free group generated by $R$ and let
$\varrho\colon F_R\to N\subset F$ be the homomorphism defined by elements of
$R$ as words on $X$.
A relation among the relators is given by an element $s\in F_R$ such that
$\varrho(s)$ is the identity element of $N$.
Any relation among the relators, say $s$, defines a map
$\tilde s\colon G^R\to G$ such that $\tilde s\circ\br\colon G^X\to G$
is the constant map taking value of the identity element.
Clearly, the image of $\br$ is contained in $\tilde s^{-1}(e)$.
We show that the subspace $\im(\ev_R)\subset\fg^R$, which is isomorphic
to $\Hom(N,\fg_\Adp)^\PI$, is contained in the formal tangent space of
$\tilde s^{-1}(e)$ at $(e,\dots,e)\in G^R$, and therefore it can not be the
whole space $\fg^R$ if there are non-trivial relations among the relators in
the presentation of $\PI$.

\begin{lem}
If $s$ is a relation among the relators, then
$\im(\ev_R)\subset\ker(d\tilde s)_{(e,\dots,e)}$.
\end{lem}

\begin{proof}
Note that both $N$ and $F_R$ acts trivially on the module $\fg_\Adp$.
Under the identifications $Z^1(N,\fg_\Adp)^\PI\cong\Hom(N,\fg_\Adp)^\PI$ and
$Z^1(F_R,\fg)\cong\fg^R$, the map $\ev_R\colon\Hom(N,\fg_\Adp)^\PI\to\fg^R$
intertwines with the pullback map
$\varrho^*\colon Z^1(N,\fg_\Adp)^\PI\to Z^1(F_R,\fg)$.
So for any $\beta\in\Hom(N,\fg_\Adp)^\PI$, we deduce from \eqref{eqn:diff} that
\[ d\tilde s(\ev_R(\beta))=(\varrho^*\tilde\beta)(s)=\tilde\beta(\varrho(s))
   =\tilde\beta(e)=0,  \]
which verifies the result.
\end{proof}

\subsection{Fox calculus, the relation module and the free resolution}
\label{sec:fox}
In this subsection, using additional algebraic concepts, we provide another
derivation of Theorem~\ref{thm:rank}, together with further understandings
of the condition $\Hom(N,\fg)^\PI\cong\fg^R$ there.

Let $w$ be a word in the free group $F$ on the finite set $X=\{x_1,\dots,x_d\}$
as in \S\ref{sec:smooth-r}.
If $G$ is a Lie group, the differential \eqref{eqn:diff} of the map
$\tilde w\colon G^X\to G$ can also be expressed in terms of the Fox derivatives
\cite{F}, which we now recall.
A {\em derivation} on $\bZ F$ is a $\bZ$-linear map $\del\colon\bZ F\to\bZ F$
such that $\del(uv)=\del(u)\eps(v)+u\del(v)$ for all $u,v\in\bZ F$, where
$\eps\colon\bZ F\to\bZ$ is the {\em augmentation map}, i.e.,
$\eps(u)=\sum_{k=1}^mn_k$ if $u=\sum_{k=1}^mn_kw_k$, $w_k\in F$, $n_k\in\bZ$.
{\em Fox derivatives} $\pdr_i$ ($i=1,\dots,d$) are derivations on $\bZ F$
defined by $\pdr_i(x_j)=\del_{ij}$ ($i,j=1,\dots,d$).
The differential $d\tilde w$ in \eqref{eqn:diff} can be expressed in terms of
Fox derivatives.
We have (cf.~\cite{G,LM}), for $\gam=(\gam_i)_{i=1,\dots,d}\in\fg^X$,
\begin{equation}\label{eqn:fox}
(d\tilde w)_\phi(\gam)=\sum_{i=1}^d\Ad_{\phi(\pdr_iw)}\gam_i.
\end{equation}

The Abelianisation $\bar N:=N/[N,N]$ of $N$ is called the {\em relation module}
of the presentation $\PI=F/N$ \cite{L62}.
It is a $\bZ\PI$-module: there is a $\PI$-action on $\bar N$ because $F$ acts
on $N$ by conjugation and hence on $\bar N$, while its subgroup $N$ acts
trivially on $\bar N$.
Clearly, $\Hom(\bar N,\fg)=\Hom(N,\fg)$ and
$\Hom(\bar N,\fg_\Adp)^\PI=\Hom(N,\fg_\Adp)^\PI$.

There is a resolution of $\bZ$ by free $\bZ\PI$-modules
\begin{equation}\label{eqn:resol}
\cdots\stackrel{d_3}\longrightarrow M_2\stackrel{d_2}\longrightarrow M_1
\stackrel{d_1}\longrightarrow M_0\stackrel{d_0}\longrightarrow\bZ,
\end{equation}
where $M_0=\bZ\PI$, $d_0=\eps$ is the augmentation map,
$M_1=\bigoplus_{i=1}^d\bZ\PI\,\hat x_i$ has a basis $\{\hat x_i\}$ in 1-1
correspondence with $X$, $d_1(\hat x_i)=[x_i-1]_\PI$,
$M_2=\bigoplus_{j=1}^q\bZ\PI\,\hat r_j$ has a basis $\{\hat r_j\}$ in 1-1
correspondence with $R$, $d_2(\hat r_j)=\sum_{i=1}^d[\pdr_ir_j]_\PI\,\hat x_i$
(see for example \cite[\S II.3]{LS}).
Here we denote by $[u]_\PI$ the image of $u\in\bZ F$ in $\bZ\PI$.
It can be shown (see for example \cite[\S11.5, Theorem~1]{J}) that
$\ker(d_1)\cong\bar N$, the relation module.

The group cohomology $H^\bullet(\PI,\fg_\Adp)$ is the cohomology of the
cochain complex
\[ \Hom_{\bZ\PI}(M_0,\fg_\Adp)=\fg\stackrel{d_1^\vee}\longrightarrow
\Hom_{\bZ\PI}(M_1,\fg_\Adp)=\fg^X\stackrel{d_2^\vee}\longrightarrow
\Hom_{\bZ\PI}(M_2,\fg_\Adp)=\fg^R\stackrel{d_3^\vee}\longrightarrow\cdots \]
dual to \eqref{eqn:resol} with values in $\fg_\Adp$.
The maps are $d_1^\vee=(\Ad_{\phi(x_i)}-1)_{i=1,\dots,d}$,
$d_2^\vee(\gam)=\big(\sum_{i=1}^d\Ad_{\phi(\pdr_ir_j)}
\gam_i\big)_{j=1,\dots,q}$ for $\gam=(\gam_i)_{i=1,\dots,d}\in\fg^X$.
By \eqref{eqn:fox}, we have $d_2^\vee=(d\br)_\phi$, and therefore
\[\rank(d\br)_\phi=\rank(d_2^\vee)=\dim\ker(d_3^\vee)-\dim H^2(\PI,\fg_\Adp).\]
By the exact sequence $\cdots\to M_3\stackrel{d_3}\longrightarrow M_2
\stackrel{d_2}\longrightarrow\ker(d_1)\to0$, we get $\ker(d_3^\vee)\cong
\Hom_{\bZ\PI}(\ker(d_1),\fg_\Adp)=\Hom(\bar N,\fg_\Adp)^\PI$.
Hence $\dim\ker(d_3^\vee)=\dim\Hom(\bar N,\fg_\Adp)^\PI$, and \eqref{eqn:rank}
in Theorem~\ref{thm:rank} is recovered.

As remarked before, $\Hom(N,\fg_\Adp)^\PI\cong\fg^R$ is equivalent to the
surjectivity of the map $\ev_R\colon\Hom(N,\fg_\Adp)^\PI\to\fg^R$.
This is clearly satisfied if $M_3=0$ in the resolution \eqref{eqn:resol}, which
implies that the cohomological dimension of $\PI$ is not greater than $2$.
In fact, $M_3$ should be thought of the module generated by the relations
among the relators.
In the special case when the presentation has a single relator \cite{L50},
we show that $\ev_R$ is always surjective.
Examples of groups whose presentations have a single relator are fundamental
groups of compact orientable or non-orientable surfaces.
Fox calculus was used in \cite{G} to study the character varieties of the
fundamental groups of orientable surfaces; we will apply Fox calculus to the
situation when the surfaces are non-orientable.
See \S\ref{sec:surf} for details.

\begin{pro}\label{pro:one}
Let $\PI$ be a group generated by a finite set $X$ with a single relator $r$.
Let $G$ be a connected Lie group with Lie algebra $\fg$.
Then $r$ defines a map $\tilde r\colon G^X\to G$ and
$\Hom(\PI,G)=\tilde r^{-1}(e)$.
If $\phi\in\Hom(\PI,G)$, then $\coker(d\tilde r)_\phi\cong H^2(\PI,\fg_\Adp)$.
If $\dim H^2(\PI,\fg_\Adp)$ is minimal at $\phi$, then $\phi$ is a smooth point
on $\Hom(\PI,G)$ and $T_\phi\Hom(\PI,G)\cong Z^1(\PI,\fg_\Adp)$.
The smooth part of $\Hom(\PI,G)$ is of dimension
\[  (|X|-1)\dim G+\dim H^2(\PI,\fg_\Adp).   \]
\end{pro}

\begin{proof}
We have $\PI=F/N$, where $F$ is the free group generated by $X$ and $N$
is the normal closure of $R:=\{r\}$ in $F$.
If $r$ is not a proper power of any element in $F$, then $M_3=0$ in
\eqref{eqn:resol} and $\Hom(N,\fg_\Adp)^\PI\cong\fg$.
If $r=s^m$ for some $s\in F$ and $m\ge2$, then there is a short exact sequence
$0\to\bZ\PI([s]_\PI-1)\to\bZ\PI\to\bar N\to0$ of $\bZ\PI$-modules (see for
example \cite[Proposition~2.4.19(b)]{CZ}).
Taking the dual sequence (with values in $\fg$), we obtain an exact sequence
\[ 0\to\Hom(\bar N,\fg_\Adp)^\PI\to\Hom_{\bZ\PI}(\bZ\PI,\fg_\Adp)
   \to\Hom_{\bZ\PI}(\bZ\PI([s]_\PI-1),\fg_\Adp)\to\cdots. \]
We show that $\Hom_{\bZ\PI}(\bZ\PI([s]_\PI-1),\fg_\Adp)=0$.
In fact, if $f\in\Hom_{\bZ\PI}(\bZ\PI([s]_\PI-1),\fg_\Adp)$, then for all
$u\in\bZ\PI$, $f(u([s]_\PI-1))=f(u[s]_\PI\cdot1)-f(u\cdot1)=f(1)-f(1)=0$
by $\PI$-invariance.
Consequently,
$\Hom(\bar N,\fg_\Adp)^\PI\cong\Hom_{\bZ\PI}(\bZ\PI,\fg_\Adp)\cong\fg$.
The results then follow from Corollary~\ref{cor:sm-hom}.
\end{proof}

\subsection{Smooth points on the character variety}\label{sec:luna}
Let $\PI$ be a finitely generated group and $G$ be a complex reductive Lie
group with Lie algebra $\fg$.
A homomorphism $\phi\in\Hom(\PI,G)$ is {\em reductive} if the image
$\phi(\PI)$ is contained in the Levi subgroup of any parabolic subgroup
containing $\phi(\PI)$.
Let $\Hom^\red(\PI,G)$ be the set of reductive homomorphisms equipped with
the subset topology.
The action of $G$ on $\Hom(\PI,G)$ by conjugation on $G$ preserves the
subspace $\Hom^\red(\PI,G)$.
The quotient space $\cR(\PI,G):=\Hom^\red(\PI,G)/G$ is the {\em character
variety} (or the {\em representation variety}) of $\PI$, and we denote by
$[\phi]$ the point in $\cR(\PI,G)$ represented by $\phi\in\Hom^\red(\PI,G)$.
Since $\phi\in\Hom(\PI,G)$ is reductive if and only if the orbit $G\cdot\phi$
is closed in $\Hom(\PI,G)$ \cite[Thm 30]{Si}, the quotient topology on
$\cR(\PI,G)$ is Hausdorff.
That $\phi$ is reductive is equivalent to the complete reducibility of the
representation $\Adp$ of $\PI$ on $\fg$.
In fact, since $\Hom(\PI,G)\subset G^X$, the latter is also equivalent to
the closedness of the orbit $G\cdot\phi$ \cite[Thm~11.4]{Ri}.

Whereas Corollary~\ref{cor:sm-hom} and Proposition~\ref{pro:one} give
sufficient conditions for the smoothness of $\Hom(\PI,G)$ at a point $\phi$,
the smoothness of the character variety $\cR(\PI,G)$ at $[\phi]$ has
further requirements due to possible quotient singularities.
A crucial step of establishing the local structure of a quotient space like
$\cR(\PI,G)$ is the slice theorem.

Suppose $G$ acts algebraically on an affine variety $X$.
Let $X\sslash G$ be the quotient by $G$ of the subset of $x\in X$ such
that $G\cdot x$ is closed, equipped with the quotient topology which is
Hausdorff, and let $[x]$ be the point in $X\sslash G$ from $x$.
A {\em slice} at $x\in X$ is a locally closed affine subvariety $S$ containing
$x$ and invariant under the stabiliser $G_x$ of $x$ such that the $G$-action
induces a $G$-equivariant homeomorphism from $G\times_{G_x}S$ to a
neighbourhood of the orbit $G\cdot x$ in $X$.
It is well known that a slice always exists for compact group actions.
Luna's slice theorem \cite{Lu} (see \cite{Dr} for an introduction) is about
the existence of a slice in the algebro-geometric context.
A subset $U\subset X$ is {\em saturated} if for all $x\in U$ and $y\in X$,
$\overline{G\cdot x}\cap\overline{G\cdot y}\ne\emptyset$ implies $y\in U$.
A saturated subset is $G$-invariant and a saturated open subset in $X$
descends to an open subset in $X\sslash G$.
Luna's slice theorem states that if $G\cdot x$ is closed in $X$, then there is
a slice $S$ at $x$ such that the $G$-morphism $G\times_{G_x}S \to X$ induced
by the action of $G$ on $X$ maps surjectively onto a saturated open subset
$U\subset X$ and that the map $G\times_{G_x}S\to U$ is strongly \'etale.
Recall from \cite[Definition 4.14]{Dr} that \emph{strongly \'etale} implies
that the induced map $S\sslash G\to U\sslash G$ is \'etale, which implies
that the underlying analytic spaces are locally isomorphic in the complex
topology \cite[\S III.5, Corollary~2]{Mu}.
Consequently, there is a homeomorphism between an open neighbourhood of
$\{x\}/G_x$ in $S\sslash G_x$ and an open neighbourhood of $[x]$ in
$X\sslash G$.
If in addition $X$ is smooth at $x$, then the slice $S$ can be chosen smooth
and we have $T_xX=T_x(G\cdot x)\oplus T_xS$.
If furthermore $G_x$ is trivial or minimal at $x$, then $X\sslash G$ is smooth
at $[x]$, and there is a diffeomorphism from a neighbourhood of $\{x\}/G_x$ in
$S/G_x$ to a neighbourhood of $[x]$ in $X\sslash G$.

Luna's slice theorem applies to our case because $G$ is a complex reductive
Lie group; it was applied to study the singularities of the character
varieties of finitely generated free groups \cite{FL}.

\begin{thm}\label{thm:var}
Let $\PI$ be a finitely presented group and $G$ be a complex reductive Lie
group.
Then for any reductive homomorphism $\phi\in\Hom(\PI,G)$, there is a slice
$S$ at $\phi$ of the $G$-action on $\Hom(\PI,G)$ such that there is a
$G$-equivariant local homeomorphism from $G\times_{G_\phi}S$ to a saturated
open subset in $\Hom(\PI,G)$.
Consequently, there is a homeomorphism from a neighbourhood of
$\{\phi\}/G_\phi$ in $S\sslash G_\phi$ to a neighbourhood of $[\phi]$ in
$\cR(\PI,G)$.
If in addition $\phi$ is a smooth point in $\Hom(\PI,G)$, then $S$ can be
chosen smooth and $T_\phi\Hom(\PI,G)=T_\phi(G\cdot\phi)\oplus T_\phi S$.
If furthermore $\phi$ is good, then $[\phi]$ is a smooth point in
$\cR(\PI,G)$, and there is a diffeomorphism from a neighbourhood of $\phi$
in $S$ to a neighbourhood of $[\phi]$ in $\cR(\PI,G)$.
\end{thm}

\begin{proof}
The reductive group $G$ is an affine variety and it acts algebraically on the
affine variety $\Hom(F,G)=G^X$, where $F$ is the free group on the finite set
$X$ of generators of $\PI$.
Since $\PI$ is finitely presented, $\Hom(\PI,G)$ is an affine subvariety in
$G^X$ defined by finitely many relators in $R=\{r_j\}_{j=1,\dots q}$.
The results then follow from Luna's slice theorem and its smooth version
reviewed above.
\end{proof}

The Lie algebra of the stabiliser subgroup $G_\phi$ is
$\fg^\PI=H^0(\PI,\fg_\Adp)$.
Let $\Ad'$ be the adjoint action of $G$ on $\fg'=[\fg,\fg]$ and let
$\fg'_{\Ad'\circ\phi}$ be the vector space $\fg'$ equipped with the
$\PI$-module structure $\Ad'\circ\phi$.
Since $H^i(\PI,\fg_\Adp)=H^i(\PI,\fg'_{\Ad'\circ\phi})\oplus H^i(\PI,Z(\fg))$,
the dimension of $H^i(\PI,\fg_\Adp)$ is minimal if
$H^i(\PI,\fg'_{\Ad'\circ\phi})=0$ for any $i\ge0$.
Let $\Hom_\circ(\PI,G)$ be the set of $\phi\in\Hom^\red(\PI,G)$ such that
$G_\phi=Z(G)$ and the quantity
$\dim_\bC\Hom(N,\fg)^\PI-\dim_\bC H^2(\PI,\fg_\Adp)$ reaches its maximum.
Then $H^0(\PI,\fg_\Adp)=Z(\fg)$ and the quotient
$\cR_\circ(\PI,G):=\Hom_\circ(\PI,G)/G$ is in the smooth part of $\cR(\PI,G)$.
Sometimes $\cR_\circ(\PI,G)$ does coincide with the set of smooth points
(Example~\ref{eg:sing}).
But there can be smooth points outside $\cR_\circ(\PI,G)$, which can even be
empty (Examples~\ref{eg:simple} and \ref{eg:h2}).

\begin{cor}\label{cor:var}
Suppose $\PI=F/N$ is a finitely presented group, where $F$ is a free group
generated by a finite set $X$ and $N$ is a normal subgroup in $F$ generated
in $F$ by a finite set of relators $R$.
Let $\phi\in\Hom_\circ(\PI,G)$.
Then there is a neighbourhood of $[\phi]$ in $\cR_\circ(\PI,G)$ diffeomorphic
to a neighbourhood of $0$ in $H^1(\PI,\fg_\Adp)$, and
\[  \dim_\bC\cR_\circ(\PI,G)=(|X|-1)\dim_\bC G+\dim_\bC Z(\fg)
    -\dim_\bC\Hom(N,\fg_\Adp)^\PI+\dim_\bC H^2(\PI,\fg_\Adp).  \]
In particular, if there exists $\phi\in\Hom_\circ(\PI,G)$ such that
$\Hom(N,\fg_\Adp)^\PI\cong\fg^R$ and $H^2(\PI,\fg'_{\Ad'\circ\phi})=0$, then
\[  \dim_\bC\cR_\circ(\PI,G)=(|X|-|R|-1)\dim_\bC G+\dim_\bC Z(\fg)
    +\dim_\bC H^2(\PI,Z(\fg)).  \]
\end{cor}

\begin{proof}
Recall that if we identify $\fg^X$ with $Z^1(F,\fg_\Adp)$, then
$\ker(d\br)_\phi=Z^1(\PI,\fg_\Adp)$.
Moreover, $T_\phi(G\cdot\phi)=B^1(\PI,\fg_\Adp)$ because for any $\xi\in\fg$,
the corresponding vector field on $G^X$ at $(\phi(x_i))_{i=1,\dots,d}$ is
$(\xi-\Ad_{\phi(x_i)}\xi)_{i=1,\dots,d}$, which is in $B^1(F,\fg_\Adp)$;
it is in fact in $B^1(\PI,\fg_\Adp)$ because $\phi(N)=1$.
The local diffeomorphism from $\cR_\circ(\PI,G)$ to $H^1(\PI,\fg_\Adp)$ follows
from $T_\phi S\cong Z^1(\PI,\fg_\Adp)/B^1(\PI,\fg_\Adp)=H^1(\PI,\fg_\Adp)$.
The complex dimension of $\cR_\circ(\PI,G)$ is that of $H^1(\PI,\fg_\Adp)$;
it is also equal to $\dim_\bC\Hom_\circ(\PI,G)-\dim_\bC(G\cdot\phi)$, where
$\dim_\bC\Hom_\circ(\PI,G)$ is given by Corollary~\ref{cor:sm-hom} and
$\dim_\bC(G\cdot\phi)=\dim_\bC G-\dim_\bC Z(G)$.
The dimension formulas then follow.
\end{proof}

Let $\PI=\pi_1(M)$ be the fundamental group of a compact K\"ahler manifold
$(M,\om)$ of complex dimension $m$.
The universal cover $\oM\to M$ is a principal $\PI$-bundle and it is the
pullback of the universal $\PI$-bundle $E\PI\to B\PI$ by a classifying map
$f\colon M\to B\PI$ that induces an isomorphism on the fundamental groups.
Recall that $H^k(\PI,\bC)\cong H^k(B\PI,\bC)$ for any $k\ge0$.
Given $\phi\in\Hom(\PI,G)$, for any $\al,\beta\in H^1(\PI,\fg_\Adp)$, consider
the cup product $\al\cup\beta\in H^2(\PI,\bC)\cong H^2(B\PI,\bC)$ defined using
an invariant non-degenerate symmetric bilinear form on $\fg$.
{}From $[\om]\in H^2(M,\bC)$, we get $[\om^{m-1}/(m-1)!]\in H^{2(m-1)}(M,\bC)$
and its Poincar\'e dual $\PD[\om^{m-1}/(m-1)!]\in H_2(M,\bC)$.
There is a symplectic form $\upom_\phi$ on $H^1(\PI,\fg_\Adp)$ given by
\cite{JR}
\begin{equation}\label{eqn:omp}
\upom_\phi(\al,\beta)=\bra\al\cup\beta,f_*\PD[\om^{m-1}/(m-1)!]\ket
   =\bra f^*(\al\cup\beta)\cup[\om^{m-1}/(m-1)!],[M]\ket.
\end{equation}
and it agrees with the gauge theoretic definition in \eqref{eqn:omD}.
When $M$ is a compact orientable surface (i.e.,$m=1$), the symplectic form
$\upom_\phi(\al,\beta)=\bra f^*(\al\cup\beta),[M]\ket$ was constructed in
\cite{G}.

\subsection{Character variety of an index-$2$ subgroup}\label{sec:surf}
Let $\PI$ be a finitely generated group containing a normal subgroup $\tPI$
of finite index.
Then $\tPI$ is also finitely generated.
If $\phi\in\Hom(\PI,G)$, let $\tphi$ be its restriction to $\tPI$.
Then the finite group $\Gam:=\PI/\tPI$ acts on $H^k(\tPI,\fg_\Adtp)$ and
$H^k(\tPI,\fg_\Adtp)^\Gam\cong H^k(\PI,\fg_\Adp)$.
To verify this, we can use the Lyndon-Hochschild-Serre spectral sequence
\cite{L48,HS}, with $E_2^{pq}=H^p(\Gam,H^q(\tPI,\fg_\Adtp))$, converging
to $H^k(\PI,\fg_\Adp)$.
Since $\Gam$ is a finite group and $H^q(\tPI,\fg_\Adtp)$ for all $q\ge0$
are divisible, we have $E_2^{pq}=0$ for all $p>0,q\ge0$, and thus
$H^k(\PI,\fg_\Adp)=E_2^{0k}=H^k(\tPI,\fg_\Adtp)^\Gam$ for all $k\ge0$.

Though the general case is quite straightforward, we specialise to the case
where $\tPI$ is an index-$2$ subgroup in $\PI$.
Since $\Gam=\bZ_2$, there is an involution $\tau$ on $H^k(\tPI,\fg_\Adtp)$
and $H^k(\tPI,\fg_\Adtp)^\tau\cong H^k(\PI,\fg_\Adp)$ for all $k\ge0$.
The action of $\tau$ can be understood as follows.
Pick any $c\in\PI\backslash\tPI$.
Then $c$ acts on $\tPI$ by $\Ad_c$ and on $\fg$ by $\Ad_{\phi(c)}$.
Let $\tau$ be the induced action of $c$ on $H^k(\tPI,\fg_\Adtp)$;
it does not depend on the choice of $c$.
Since $c^2\in\tPI$ acts on the cohomology groups trivially, $\tau$ is an
involution and, consequently,
\[  H^k(\tPI,\fg_\Adtp)=
    H^k(\tPI,\fg_\Adtp)^\tau\oplus H^k(\tPI,\fg_\Adtp)^{-\tau},    \]
where $H^k(\tPI,\fg_\Adtp)^{\pm\tau}$ are, respectively, the eigenspaces on
which $\tau$ takes eigenvalues $\pm1$.
As explained above, $H^k(\PI,\fg_\Adp)\cong H^k(\tPI,\fg_\Adtp)^\tau$.

If $\PI=\pi_1(M)$ and $\tPI=\pi_1(\tM)$, where $M$ is a closed non-orientable
surface and $\tM$ is its the oriented double cover, then $\tPI$ can be
identified as an index-$2$ subgroup of $\PI$.
Choosing an invariant non-degenerate symmetric bilinear form on $\fg$,
there is a non-degenerate pairing, for each $k=0,1,2$,
\begin{equation}\label{eqn:pair}
H^k(\tPI,\fg_\Adtp)\times H^{2-k}(\tPI,\fg_\Adtp)\to H^2(\tPI,\bC)\cong\bC.
\end{equation}
When $k=1$, this is the symplectic form \eqref{eqn:omp} on
$H^1(\tPI,\fg_\Adtp)$.
Since the pairing \eqref{eqn:pair} changes sign under $\tau$, it induces
non-degenerate pairings
\[ H^k(\tPI,\fg_\Adtp)^\tau\times H^{2-k}(\tPI,\fg_\Adtp)^{-\tau}\to\bC. \]
On the other hand, $H^1(\PI,\fg_\Adp)\cong H^1(\tPI,\fg_\Adtp)^\tau$ is a
Lagrangian subspace in $H^1(\tPI,\fg_\Adtp)$.

Another important consequence is that for the fundamental group $\PI$ of
a non-orientable surface, $H^2(\PI,\fg_\Adp)$ is not isomorphic to
$H^0(\PI,\fg_\Adp)$.
Indeed, \eqref{eqn:pair} implies that $H^2(\tPI,\fg_\Adtp)$ is isomorphic to
$H^0(\tPI,\fg_\Adtp)=\fg^\tPI$.
Whereas
$H^0(\PI,\fg_\Adp)\cong H^0(\tPI,\fg_\Adtp)^\tau=\fg^\tPI\cap\ker(\Ad_c-1)$,
we have
$H^2(\PI,\fg_\Adp)\cong H^2(\tPI,\fg_\Adtp)^\tau
\cong H^0(\tPI,\fg_\Adtp)^{-\tau}=\fg^\tPI\cap\ker(\Ad_c+1)$, and hence
\begin{equation}\label{eqn:decomp}
H^0(\tPI,\fg_\Adtp)\cong H^2(\tPI,\fg_\Adtp)\cong
H^0(\PI,\fg_\Adp)\oplus H^2(\PI,\fg_\Adp).
\end{equation}
Therefore $H^2(\tPI,\fg_\Adtp)=0$ implies $H^2(\PI,\fg_\Adp)=0$, but the
converse is not true (see Appendix).

Suppose $\PI$ is the fundamental group of an orientable surface and
$\phi\in\Hom(\PI,G)$.
In \cite{G}, Fox calculus was used to show that the rank of the map
$(d\br)_\phi$ equals the codimension of centraliser of $\phi$.
In this case, the Lie algebra of the centraliser $H^0(\PI,\fg_\Adp)$ is
isomorphic to the second cohomology group $H^2(\PI,\fg_\Adp)$.
If $\PI$ is the fundamental group of a non-orientable surface, we showed
that the codimension of $\im(d\br)_\phi$ equals the dimension of
$H^2(\PI,\fg_\Adp)$ (Proposition~\ref{pro:one}), but $H^2(\PI,\fg_\Adp)$ is
not isomorphic to $H^0(\PI,\fg_\Adp)$ as explained above.
We can verify this in another way by giving an explicit formula for
$H^2(\PI,\fg_\Adp)$ using Fox calculus.
If $M$ is the connected sum of $h$ copies of $\bR P^2$, then its fundamental
group $\PI=\pi_1(M)$ is generated by $x_1,\dots,x_h$ subject to one relation
$\prod_{i=1}^hx_i^2=e$.

\begin{pro}\label{pro:H02}
Let $\PI$ be the fundamental group of a non-orientable surface as above and
let $G$ be a complex reductive Lie group.
Then for any $\phi\in\Hom(\PI,G)$, we have
\[ H^0(\PI,\fg_\Adp)\cong\bigcap_{i=1}^h\ker(\Ad_{\phi(x_i)}-1),\quad
   H^2(\PI,\fg_\Adp)\cong\bigcap_{i=1}^h\ker(\Ad_{\phi(x_i)}+1).      \]
\end{pro}

\begin{proof}
The first equality is obvious because $H^0(\PI,\fg_\Adp)=\fg^\PI$.
Since the presentation of $\PI$ has a single relator $r=\prod_{i=1}^hx_i^2$,
we have, by Proposition~\ref{pro:one},
$H^2(\PI,\fg_\Adp)\cong\coker(d\tilde r)_\phi$, which can be calculated by
Fox derivatives as follows.
By \eqref{eqn:fox}, we obtain
$\im(d\tilde r)_\phi=\sum_{k=1}^h\im(\Ad_{\phi(\pdr_kr)})$, where the Fox
derivatives are $\pdr_kr=\big(\prod_{i=1}^{k-1}x_i^2\big)(x_k+1)$ for
$1\le k\le h$.
Choose a non-degenerate invariant symmetric bilinear form on $\fg$.
By the identity $\im(S(T+1))^\perp\cong\ker((T+1)S^{-1})$ for (complex)
orthogonal transformations $S$ and $T$ on $\fg$, we get
\[ \coker(d\tilde r)_\phi\cong\im(d\tilde r)_\phi^\perp=\bigcap_{k=1}^h
\ker\Big((\Ad_{\phi(x_k)}+1)\prod_{i=1}^{k-1}\Ad_{\phi(x_{k-i})}^{-2}\Big), \]
which agrees with our result on $H^2(\PI,\fg_\Adp)$ upon a simplification.
\end{proof}

The quotient map $\PI\to\bZ_2$ is given by sending any
$\prod_{k=1}^lx_{i_k}^{m_k}$ (with $m_k\in\bZ$) in $\PI$ to
$\sum_{k=1}^lm_k\;(\!\!\!\!\mod 2)$.
Its kernel $\tPI$ is generated by $x_ix_j$ ($1\le i,j\le h$), and we can
choose any $x_i$ as $c\in\PI\backslash\tPI$.
It is clear that
\[ \bigcap_{i,j=1}^h\ker(\Ad_{\phi(x_i)\phi(x_j)}-1)=
   \bigcap_{i=1}^h\ker(\Ad_{\phi(x_i)}-1)\oplus
   \bigcap_{i=1}^h\ker(\Ad_{\phi(x_i)}+1),   \]
as subspaces of $\fg$, verifying the decomposition \eqref{eqn:decomp} using
Fox calculus.

Proposition~\ref{pro:H02} provides an explicit formula for the second
cohomology group $H^2(\PI,\fg_\Adp)$ of the fundamental group $\PI$ of a
non-orientable surface $M$.
If we use other presentations of $\PI$ (for example, the one used in
\cite{HL,HWW}), similar explicit formulas for $H^2$ exist, also verifying
\eqref{eqn:decomp}.
With the presentation here, if there is no non-zero vector $\xi\in\fg$
satisfying $\Ad_{\phi(x_i)}\xi=-\xi$ for all $1\le i\le h$, then $\phi$
is a smooth point on $\Hom(\PI,G)\subset G^h$, and
$T_\phi\Hom(\PI,G)\cong\fg^{h-1}$.
If furthermore the stabiliser of $\phi$ is $Z(G)$, then $[\phi]$ is in the
smooth part $\cR_\circ(\PI,G)$ of the character variety, whose complex
dimension is $(h-2)\dim_\bC G+\dim_\bC Z(G)$.
Since the double cover $\tM$ has genus $h-1$, the complex dimension of
$\cR_\circ(\PI,G)$ is half of that of $\cR_\circ(\tPI,G)$.
This agrees with the statement that $H^1(\PI,\fg_\Adp)$ is a Lagrangian
subspace of $H^1(\tPI,\fg_\Adtp)$.

\section{Relation of the two approaches}\label{sec:compare}
Let $M$ be a closed manifold and let $G$ be a reductive complex Lie group.
We denote by $\cA^\fla(P)$ the space of flat connections on a principal
$G$-bundle $P$ over $M$.
Let $\cA^\fla(M,G)$ be the union of $\cA^\fla(P)$ for all topological types
of $P$ and let $\cG(M,G)$ be the group of gauge transformations that is
$\cG(P)$ on $\cA^\fla(P)$.
We equip $\cA^\fla(M,G)$ and $\cG(M,G)$ with the ``smooth topology'', in
which convergence means that in the Sobolev norm $\|\cdot\|_{2,k}$ for every
$k$.

Choose base points $x\in M$ and $p$ in the fibre over $x$, and let
$\PI=\pi_1(M,x)$.
Then there is a holonomy map $\hol_p\colon\cA(M,G)\to\Hom(\PI,G)$, sending
a flat connection $A$ to the homomorphism
$\hol_p(A)\colon[\al]\mapsto\hol_\al(A)\in G$,
where $\al$ is a based loop in $(M,x)$ representing $[\al]\in\PI$ and the
parallel transport along $\al$ sends $p$ to $p\hol_\al(A)$.
So the map $\hol_p$ is obtained from solving a family of first order
differential equations parametrised by $\cA(M,G)$.
It is continuous and equivariant with respect to the actions of $\cG(M,G)$
on $\cA^\fla(M,G)$ and $G$ on $\Hom(\PI,G)$.
Moreover, it yields a 1-1 correspondence between the $\cG(M,G)$-orbits
in $\cA^\fla(M,G)$ and the $G$-orbits in $\Hom(\PI,G)$ \cite[\S5]{GM}.

By the definitions of reductive flat connections in $\cA^\fla(M,G)$ and
reductive homomorphisms in $\Hom(\PI,G)$, it is obvious that a flat connection
$A$ is reductive if and only if the homomorphism $\phi:=\hol_p(A)$ is
reductive.
Recall that $\phi$ is reductive if and only if its $G$-orbit $G\cdot\phi$ is
closed in $\Hom(\PI,G)$ (cf.~\S\ref{sec:luna}).
We claim that a flat connection $A$ is reductive if and only if its orbit
under the group of gauge transformations is closed.
In fact, if $A$ is reductive, then so is $\phi$ and the orbit $G\cdot\phi$
is closed.
{}From the above properties of the map $\hol_p$, it is easy to see that the
orbit of $A$ is also closed.
Conversely, if a flat connection $A$ is not reductive, then there is sequence
of gauge transformations of $A$ that goes to a limit outside the orbit
\cite[Prop 3.2]{Co}, and the orbit of $A$ is not closed.

So $\hol_p$ induces a 1-1 continuous map from the moduli space $\cM^\fla(M,G)$
to the character variety $\cR(\PI,G)$, since both spaces are constructed from
taking quotients of reductive objects.
On the other hand, from a homomorphism $\phi\in\Hom(\PI,G)$, we can construct
in a standard way a $G$-bundle $\oM\times_\phi G$ (where $\oM$ is the
universal cover of $M$) over $M$ with a flat connection whose holonomy is
$\phi$.
This gives the inverse map from $\cR(\PI,G)$ to $\cM^\fla(M,G)$, which is also
continuous.
Consequently, the spaces $\cM^\fla(M,G)$ and $\cR(\PI,G)$ are homeomorphic.

Suppose $\tM\to M$ is a regular covering and $\Gam$ is the group of deck
transformations.
Let $\PI=\pi_1(M)$ and $\tPI=\pi_1(\tM)$ as before.
Then there is a fibration $E\Gam\times_\Gam\tM\to B\Gam$ whose fibre is $\tM$
and whose total space, being the Borel construction of $\tM$ with the free
$\Gam$-action, is homotopic to $\tM/\Gam=M$.
Let $G$ be a complex reductive Lie group and let $P\to M$ be a principal
$G$-bundle with a flat connection.
Then there is a $G$-bundle on the total space $E\Gam\times_\Gam\tM$ whose
restriction to the fibre $\tM$ is the pull-back $\tP\to\tM$ of $P$.
There is a spectral sequence with
$\cE_2^{pq}=H^p(\Gam,H^q(\tM,\ad\tP))$ converging to $H^k(M,\ad P)$.
This is called the Cartan-Leray spectral sequence associated to the regular
covering $\tM\to M=\tM/\PI$ \cite[\S VII.7]{B}, and it is the gauge theoretic
analog of the Lyndon-Hochschild-Serre spectral sequence in group cohomology
associated to $1\to\tPI\to\PI\to\Gam\to1$ with
$E_2^{pq}=H^p(\Gam,H^q(\tPI,\fg_\Adtp))$, converging to $H^k(\PI,\fg_\Adp)$
(cf.~\S\ref{sec:surf}).
Since $H^0(\tM,\ad\tP)=H^0(\tPI,\fg_\Adtp)$, we have $\cE_2^{k0}=E_2^{k0}$
for all $k\ge0$.

As in \S\ref{sec:surf}, we apply the result to the case where $M$ is a
compact non-orientable manifold and $\tM$ is its orientation double cover.
Then $\Gam=\bZ_2$ and the non-trivial element in $\bZ_2$ acts on the
cohomology groups $H^k(\tM,\ad\tP)$ by $\tau$ (cf.~\ref{sec:non-g}).
Since $\cE_2^{pq}=0$ for $p>0,q\ge0$, we obtain
$H^k(M,\ad P)=\cE_2^{0k}=H^k(\tM,\ad\tP)^\tau$ for all $k\ge0$.
This is the gauge theoretic counterpart of the statement
$H^k(\PI,\fg_\Adp)\cong H^k(\tPI,\fg_\Adtp)^\tau$ in \S\ref{sec:surf}.
Taking $k=1$, we conclude that the (formal) tangent space of the moduli space
of flat $G$-connections on $M$ is the $\bZ_2$-invariant subspace of the
tangent space to the moduli space of flat $G$-connections on $\tM$.

Another important example is when $\tM$ is the universal cover $\oM$ of $M$.
Then $\PI=\Gam$ and the pullback of $P$ to $\oM$ is topologically the product
bundle $\oM\times G$.
We have $H^q(\oM,\ad\tP)\cong H^q(\oM)\otimes\fg_\Adp$, where the group $\PI$
acts on $H^q(\oM)$ by the pullback of deck transformations as well as on $\fg$
by the homomorphism $\phi\colon\PI\to G$ corresponding to the flat connection.
Since $H^1(\oM)=0$, we have $\cE_2^{p1}=0$ for all $p\ge0$, and
$H^1(M,\ad P)=\cE_2^{10}=H^1(\PI,\fg_\Adp)$.
This means that the (formal) tangent space of the moduli space of flat
$G$-connections on $M$ in the gauge theoretic approach
(cf.~\S\ref{sec:smooth-g}) is identical to that of the character variety
in a more algebraic approach (cf.~\S\ref{sec:luna}).

To find information on the second cohomology groups, we note that since
$\cE_\infty^{11}\subset\cE_2^{11}=0$, there is a short exact sequence
$0\to\cE_\infty^{20}\to H^2(M,\ad P)\to\cE_\infty^{02}\to0$.
Here $\cE_\infty^{20}=\cE_3^{20}=\coker(\cE_2^{01}\to\cE_2^{20})=\cE_2^{20}
=H^2(\PI,\fg_\Adp)$ and
$\cE_\infty^{02}=\ker(\cE_3^{02}\to\cE_3^{30})=\ker(\cE_2^{02}\to\cE_2^{30})
=\ker(H^2(\oM,\fg_\Adp)^\PI\to H^3(\PI,\fg_\Adp))$.
By the Hurewicz theorem, we have $H_2(\oM)\cong\pi_2(\oM)\cong\pi_2(M)$, and
the isomorphisms are equivariant under $\PI$.
Hence $H^2(\oM,\fg_\Adp)^\PI\cong\Hom(H_2(\oM),\fg_\Adp)^\PI\cong
\Hom(\pi_2(M),\fg_\Adp)^\PI$.
On the other hand, we have $\cE_2^{30}\subset H^3(M,\ad P)$ and because of
$\cE_2^{11}=0$ again, we have $\cE_2^{3,0}=\coker(\cE_2^{02}\to\cE_2^{30})$.
So there is an exact sequence (cf.~\cite[Exer.~6, page~175]{B})
\begin{equation}\label{eqn:2H2}
0\to H^2(\PI,\fg_\Adp)\to H^2(M,\ad P)\to\Hom(\pi_2(M),\fg_\Adp)^\PI\to
H^3(\PI,\fg_\Adp)\to H^3(M,\ad P).
\end{equation}
This means that the second cohomology groups $H^2(M,\ad P)$ and
$H^2(\PI,\fg_\Adp)$ that appear as obstructions to smoothness in the two
approaches (cf.~\S\ref{sec:smooth-g} and \S\ref{sec:smooth-r}) are not
identical but differ by
$\cE_\infty^{02}=\ker(\Hom(\pi_2(M),\fg_\Adp)^\PI\!\to\!H^3(\PI,\fg_\Adp))$.
However, if $\pi_2(M)=0$, then $H^2(M,\ad P)\cong H^2(\PI,\fg_\Adp)$.
This is the case if the universal cover $\oM$ is contractible, for example if
$M$ is an orientable surface of genus $g\ge1$, a non-orientable surface which
is the connected sum of $h\ge2$ copies of $\bR P^2$, a hyperbolic $3$-manifold,
or a K\"ahler manifold which is a ball quotient.

In \S\ref{sec:smooth-g}, smoothness at a point of the (gauge-theoretic) moduli
space requires, among other things, the condition $H^2(M,\ad'P)=0$, which
appears stronger than its algebraic version $H^2(\PI,\fg'_{\Ad'\circ\phi})=0$
by \eqref{eqn:2H2} as explained above.
On the other hand, by applying the implicit function theorem in
\S\ref{sec:smooth-r} and \S\ref{sec:fox}, the character variety $\cR(\PI,G)$
is smooth at points $[\phi]$ where $H^2(\PI,\fg'_{\Ad'\circ\phi})=0$ and
$\ev_R$ is surjective.
While the condition $H^2(\PI,\fg'_{\Ad'\circ\phi})=0$ is common in the two
approaches, it would be interesting to compare the additional requirements
that differ.
Fortunately, for orientable and non-orientable surfaces, these additional
conditions happen to be vacuous, though for different reasons, and we arrive
at the same condition $H^2(\PI,\fg'_{\Ad'\circ\phi})=0$ from both methods.
If this way, the subset $\cM^\fla_\circ(P)$ (cf.~\S\ref{sec:smooth-g}) in the
smooth part of the moduli space $\cM^\fla(M,G)$ coincides with its counterpart
$\cR_\circ(\PI,G)$ (cf.~\S\ref{sec:luna}) in the smooth part of the character
variety $\cR(\PI,G)$, and the homeomorphism between $\cM^\fla(M,G)$ and
$\cR(\PI,G)$ restricts to a diffeomorphism between $\cM^\fla_\circ(M,G)$ and
$\cR_\circ(\PI,G)$.

\renewcommand{\thesection}{A}
\section*{Appendix. Some examples}
Let $G$ be a reductive Lie group.
Suppose $M$ is a non-orientable manifold and $\pi\colon\tM\to M$ is its
orientation double cover.
A flat $G$-connection $D$ on $M$ lifts to a flat connection $\tD$ on $\tM$.
Alternatively, $D$ determines a homomorphism $\phi\colon\PI\to G$, where
$\PI=\pi_1(M)$.
The homomorphism $\tphi\colon\tPI\to G$ from $\tD$ is the restriction of
$\phi$ to $\tPI=\pi_1(\tM)$, which is an index-$2$ subgroup of $\PI$.
We compare various conditions on $D$ (or $\phi$) and $\tD$ (or $\tphi$).

First, $D$ (or $\phi$) is {\em simple} if the Lie algebra of its stabiliser
is that of $Z(G)$.
The condition is equivalent to $H^0(M,\ad'P)=0$ or
$H^0(\PI,\fg_{\Ad'\circ\phi})=0$.
Clearly, $D$ (or $\phi$) is simple if $\tD$ (or $\tphi$) is so, but the
converse is not true.
In the examples below, we utilise the standard Pauli matrices
\[ \sig_1={\quad\;\;1\choose1\;\;\quad},\quad
\sig_2={\qquad\;-\ii\,\choose\ii\;\;\quad\qquad},\quad
\sig_3={1\qquad\choose\quad-1}                           \]
satisfying $\sig_i\sig_j=-\sig_j\sig_i$ ($i\ne j$) and $\sig_i^2=I_2$,
$\tr(\sig_i)=0$, $\det(\sig_i)=-1$ ($i=1,2,3$).

\begin{eg}\label{eg:simple}{\rm
Take $M=\bR P^2\#\bR P^2$ or $\PI$ generated by $x_1,x_2$
subject to a relation $x_1^2x_2^2=1$.
Let $G=\mathrm{SL}(2,\bC)$ (whose centre is finite) and $\phi$ be defined by
$\phi(x_1)=\ii\sig_1$, $\phi(x_2)=\ii\sig_2$.
Then by Proposition~\ref{pro:H02}, we obtain $H^0(\PI,\fg_\Adp)=0$ and
$H^2(\PI,\fg_\Adp)\cong\bC\sig_3$.
On the other hand, by \eqref{eqn:decomp} or by a direct calculation, we
obtain $H^0(\tPI,\fg_\Adtp)\cong H^2(\tPI,\fg_\Adtp)\cong\bC\sig_3$.
So $\phi$ is simple but $\tphi$ is not.}
\end{eg}

Following \cite{Si}, we say that $D$ (or $\phi$) is {\em irreducible} if
it is reductive and simple.
In \cite{HWW}, we proved that $D$ (or $\phi$) is reductive if and only if
$\tD$ (or $\tphi$) is so.
Thus if $\tD$ (or $\tphi$) is irreducible, then so is $D$ (or $\phi$).
The converse is not true because $\phi$ in Example~\ref{eg:simple} is indeed
reductive and hence irreducible, while $\tphi$ is not irreducible.

Using \eqref{eqn:decomp}, we also conclude that
$H^2(\tPI,\fg_{\Ad'\circ\tphi})=0$ implies $H^2(\PI,\fg_{\Ad'\circ\phi})=0$.
But the converse is not true either as shown in the following example.

\begin{eg}\label{eg:h2}{\rm
Take the same $M$ (or $\PI$) and $G$ as in Example~\ref{eg:simple} but let
$\phi$ be defined by
\[  \phi(x_1)=\exp(\pi\ii\sig_3/4)
    ={\zeta_8\qquad\;\;\choose\qquad\zeta_8^{-1}},\quad
\phi(x_2)=\exp(-\pi\ii\sig_3/4)
    ={\zeta_8^{-1}\qquad\choose\qquad\;\;\zeta_8}, \]
where $\zeta_8:=e^{\pi\ii/4}$ is the $8$th root of unity.
Then $H^2(\PI,\fg_\Adp)=0$ whereas}
\[ H^2(\tPI,\fg_\Adtp)\cong H^0(\tPI,\fg_\Adtp)\cong H^0(\PI,\fg_\Adp)\cong
\bC\sig_3\ne0. \]
\end{eg}

Curiously, the points $[\phi]$ represented by $\phi$ in
Examples~\ref{eg:simple} and \ref{eg:h2} are both smooth in the character
variety $\cR(\PI,g)$ despite having a non-zero $H^0$ or $H^2$.
In fact, $\cR(\PI,g)$ consists of two branches: one with $\phi(x_1)^2\in\ZG$,
parametrised by $\phi(x_1x_2)\in G$ modulo conjugation, the other being the
closure of points with $\phi(x_1)^2\not\in\ZG$,
parametrised by $\phi(x_1)^2\in G$ modulo conjugation (aside from the
remaining discrete ambiguities).
The two branches intersect at four singular points with
$\phi(x_1),\phi(x_2)\in\ZG$, but are otherwise smooth and have complex
dimension $1$.
So $\cR_\circ(\PI,G)=\emptyset$, but the generic points in $\cR(\PI,G)$ are
smooth.
Also in this case, $\tPI$ is generated by $x_1x_2$ and $x_1^2$ that
commute with each other.
Since the minimal stabiliser of a generic point in $\cR(\tPI,G)$ is a complex
$1$-dimensional subgroup in $G$, the points $[\tphi]$ in
Examples~\ref{eg:simple} and \ref{eg:h2} are both smooth in $\cR(\tPI,G)$,
despite failing to be simple or having a non-zero $H^2$.

\begin{eg}\label{eg:sing}{\rm
To produce singular points from a non-zero $H^0$ or $H^2$, we take
$M=\bR P^2\#\bR P^2\#\bR P^2\#\bR P^2$ or $\PI$ with four generator
$x_{\pm1},x_{\pm2}$ subject to a relation $x_1^2x_{-1}^2x_2^2x_{-2}^2=1$,
and $G=\mathrm{SL}(2,\bC)$.
For $t$ in the interval $(-1/4,3/4)$, we define $\phi_t\in\Hom(\PI,G)$ by
\[ \phi_t(x_{\pm1})=\pm\ii\sig_1,\quad\phi_t(x_{\pm2})
=\exp(\pm\pi t\ii\sig_2)=\cos(\pi t)I_2\pm\ii\sin(\pi t)\sig_2.\]
Then from Proposition~\ref{pro:H02} we have
\[ H^0(\PI,\fg_\Adp)=\two{\bC\sig_1}{t=0}{0}{t\ne0,}\qquad
   H^2(\PI,\fg_\Adp)=\two{\bC\sig_3}{t=1/2}{0}{t\ne1/2.}  \]
Recall the Betti numbers $b_i=\dim_\bC H^i(\PI,\fg_\Adp)$ ($i=0,1,2$).
Since $b_0-b_1+b_2=\chi(M)\dim_\bC G=-6$ is fixed, the expected dimension
$b_1$ of the tangent space of $\cR(\PI,G)$ at $[\phi_t]$ jumps at $t=0$
and $t=1/2$.
Therefore the character variety $\cR(\PI,G)$ is singular at both $[\phi_0]$
and $[\phi_{1/2}]$, due to the non-minimality of $H^0(\PI,\fg_\Adp)$ and
$H^2(\PI,\fg_\Adp)$, respectively.
This shows that the condition on $H^2$ is needed for smoothness when it is
not isomorphic to $H^0$, as in the case of orientable surface.
The points $[\tphi_0]$ and $[\tphi_{1/2}]$ are singular on $\cR(\tPI,G)$
because of the non-minimality of
$H^0(\tPI,\fg_\Adtp)\cong H^2(\tPI,\fg_\Adtp)$ at $t=0$ and $t=1/2$.
Moreover, since $\cR(\PI,G)$ is connected \cite[Thm B(d)]{BC} and so is
$\cR(\tPI,G)$ \cite{Li}, all smooth points of $\cR(\PI,G)$ and $\cR(\tPI,G)$
are in $\cR_\circ(\PI,G)$ and $\cR_\circ(\tPI,G)$, respectively.}
\end{eg}

A flat $G$-connection $D$ (or $\phi\in\Hom(\PI,G)$) is {\em good}
(cf.~\cite{JM}) if it is reductive and its stabiliser is precisely $Z(G)$.
If $D$ (or $\phi$) is good, it is irreducible, but the converse is not true.
In Example~\ref{eg:nonCI}, we will encounter a case in which the stabiliser
contains the centre $Z(G)$ as a proper subgroup of finite index.
It is also clear that $D$ (or $\phi$) is good if $\tD$ (or $\tphi$) is so.
But the converse is not true either.
In Example~\ref{eg:simple}, the Lie algebra of the stabiliser changes when $D$
is pulled back to $\tM$ (or when $\phi$ is restricted to $\tPI$).
We will construct an example in which the stabiliser changes even if its Lie
algebra remains the same when $D$ is pulled back to $\tM$.
To achieve this, we need to choose a reductive group $G$ that does not have
property CI of \cite{Si}, so that there exists an irreducible subgroup in $G$
whose centraliser is a non-trivial finite extension of $Z(G)$.

\begin{eg}\label{eg:nonCI}{\rm
We choose the simplest non-CI group $G=\mathrm{PSL}(2,\bC)$ and denote its
elements by $\pm g$, where $g\in\mathrm{SL}(2,\bC)$.
Its centre $Z(G)$ is trivial, but there is a finite Abelian subgroup
$H:=\{\pm I_2,\pm\ii\sig_1,\pm\ii\sig_2,\pm\ii\sig_3\}$, isomorphic to the Klein $4$-group, whose centraliser in $G$ is $H$ itself \cite{Si}.
We let $M=\bR P^2\#\bR P^2\#\bR P^2$ or $\PI$ be generated by $x_1,x_2,x_3$
subject to a relation $x_1^2x_2^2x_3^2=1$.
Define $\phi$ by
\[ \phi(x_1)=\pm{\zeta_8\;\;\qquad\choose\qquad\zeta_8^{-1}},\quad
\phi(x_2)=\pm{\zeta_8^{-1}\qquad\choose\qquad\;\;\zeta_8},\quad
\phi(x_3)=\pm{\qquad-\zeta_8\choose\zeta_8^{-1}\qquad}.     \]
Then the centraliser of $\phi(\PI)$ in $G=\mathrm{PSL}(2,\bC)$ is trivial.
Since $\tPI$ is generated by $x_ix_j$ ($1\le i,j\le3$), it is easy to check
that $\phi(\tPI)=H$ and hence the stabiliser of $\tphi=\phi|_\tPI$ is $H$.
So $\phi$ is good but $\tphi$ is not.}
\end{eg}

In Example~\ref{eg:nonCI}, since $H^2(\tPI,\fg_\Adtp)=0$, we have
$H^0(\PI,\fg_\Adp)=H^2(\PI,\fg_\Adp)=0$, and hence $[\phi]\in\cR_\circ(\PI,G)$
is a smooth point on $\cR(\PI,G)$.
However, $\cR(\tPI,G)$ is not smooth at $[\tphi]$.
Let $\phi_t$ ($t\in\bR$) be given by
\[ \phi_t(x_1)=\phi(x_1)\exp(\pi t\ii\sig_3),\quad
   \phi_t(x_2)=\phi(x_2)\exp(-\pi t\ii\sig_3),\quad\phi_t(x_3)=\phi(x_3) \]
and consider the path $[\tphi_t]$ in $\cR(\tPI,G)$.
Then the stabiliser of $\tphi_t$ is trivial when $t$ is near but not equal to
$0$ and is $H$ when $t=0$.
So $\cR(\tPI,G)$ has an orbifold singularity at $[\tphi_0]=[\tphi]$.

\end{document}